\input epsf

\documentclass[10pt]{amsart}

\usepackage{amscd, amsmath, amsthm, amssymb}

\def\mylabel#1{\label{#1}   \rlap{\hskip1cm\leftline{#1}}   }
\def\mylabel#1{\label{#1}   \proplabeL{#1} \hskip-3pt  }
\def\thlabel#1{\label{#1}   \proplabeL{#1} \hskip-3pt  }

\def\mylabel#1{\label{#1}}
\def\thlabel#1{\label{#1}}

\def\pd{\partial}
\def\myDelta{{\mathtt \Delta}\hspace{-0.5pt}}

\newtheorem{theorem}{Theorem}[section]
\newtheorem{lemma}[theorem]{Lemma}
\newtheorem{proposition}[theorem]{Proposition}

\theoremstyle{definition}     
\newtheorem{definition}[theorem]{Definition}

\newtheorem{remark}[theorem]{Remark}

\numberwithin{equation}{section}

\newcommand{\simrightarrow}{\smash{\mathop{\rightarrow}\limits^{\sim}}} 

\begin{document}


\title[BCOV ring and holomorphic anomaly equation]{ 
{BCOV ring and holomorphic anomaly equation}
}

\author[S. Hosono]{ Shinobu Hosono }

\address{
Graduate School of Mathematical Sciences, 
University of Tokyo, Meguro-ku, 
Tokyo 153-8914, Japan
}
\email{hosono@ms.u-tokyo.ac.jp}




\begin{abstract}
We study certain differential rings over the moduli 
space of Calabi-Yau manifolds. In the case of an elliptic curve, we 
observe a close relation to the differential ring of 
quasi-modular forms due to Kaneko-Zagier\cite{KZ}. 
\end{abstract}

\maketitle

\vskip1cm


\section{{\bf Introduction }}


Since the pioneering work by Candelas, de la Ossa, 
Green and Parkes \cite{CdOGP} in 1991, 
the theory of variation of Hodge structures  
has been one of the indispensable tools in the study of 
mirror symmetry of Calabi-Yau manifolds and its application to 
Gromov-Witten theory or enumerative geometry on Calabi-Yau manifolds. 
In particular, in the generalization due to Bershadsky, Cecotti, Ooguri 
and Vafa (BCOV) to higher genus Gromov-Witten potentials, the 
theory of variation of Hodge structures was combined with a framework which 
is called $t$-$t^*$ geometry \cite{CV}. $t$-$t^*$ geometry is a deformation 
theory of $N=2$ supersymmetric quantum field theory in two dimensions, and 
there is a natural hermitian (real) 
structure in the space of observables. BCOV 
identifies this hermitian structure with the hermitian structure over the 
moduli space of Calabi-Yau manifolds given by the Weil-Petersson 
metric, and have proposed a profound recursive relation, called {\it 
holomorphic anomaly equation}, for higher genus 
Gromov-Witten potentials. 

\vskip0.1cm

In case of dimension one, i.e. for elliptic curves, the counting 
problem and higher genus Gromov-Witten potentials have been determined 
by Dijkgraaf \cite{Di} in 1995, 
where it was remarked that BCOV theory is closely 
related to the theory of elliptic quasi-modular forms. In the same 
proceedings volume as \cite{Di}, Kaneko and Zagier have presented a general 
theory of quasi-modular forms introducing the (differential) ring of almost 
holomorphic modular forms. It is also found in \cite{HST} 
that, for an rational elliptic surface, the higher genus Gromov-Witten 
potentials are expressed by quasi-modular forms, extending the 
genus zero result by \cite{MNW},\cite{MNVW}. 

\vskip0.1cm

For Calabi-Yau threefolds, it has been expected that the 
BCOV holomorphic anomaly equation is defined over a certain differential 
ring which generalizes the  
almost holomorphic elliptic modular forms due to Kaneko and Zagier. 
Recently, in physics literatures, Yamaguchi and Yau \cite{YY} and 
later Alim and L\"ange \cite{AL} have made important developments toward 
the structure of the expected differential ring of Calabi-Yau threefolds 
(, see also \cite{ABK},\cite{GKMW}). 
In this paper, to make a  parallel argument to 
the theory due to Kaneko and Zagier, we introduce three different differential 
rings ${\mathcal R}_{BCOV}^0$, ${\mathcal R}_{BCOV}^{\Gamma}$ and 
${\mathcal R}_{BCOV}^{hol}$ over the moduli space of Calabi-Yau threefolds. 
We call these 
differential rings simply {\it BCOV rings}\footnote{Although, for simplicity, 
we use this terminology in this paper, we can call these rings 
{\it differential rings on the moduli spaces} as just defined.}. 
Our BCOV rings ${\mathcal R}_{BCOV}^\Gamma$ and 
${\mathcal R}_{BCOV}^{hol}$, should be regarded as a natural 
generalization of the ring of almost holomorphic modular forms and 
quasi-modular forms, respectively, and may be recognized in the 
original work\cite{BCOV2} and more explicitly in recent physics 
literatures\cite{YY},\cite{AL},\cite{ABK},\cite{GKMW}. 
We will introduce another form of 
the BCOV ring ${\mathcal R}_{BCOV}^0$ and observe that, with this ring, 
our parallelism to Kaneko-Zagier theory becomes complete. 

\vskip0.1cm

Here we summarize briefly the theory of quasi-modular forms due to 
Kaneko and Zagier. 
Kaneko-Zagier\cite{KZ} starts from a ring 
\begin{equation}
{\bf C}[[\tau]] \big[ \frac{1}{\tau-\bar \tau} \big] \;\;,
\mylabel{eqn:KZ-0}
\end{equation}
with $\tau$ in the upper-half plane, and the standard modular group action,  
$\tau \mapsto \frac{a \tau + b}{c \tau + d}$. Inside this large ring, one  
first considers the almost modular forms $\widehat M(\Gamma)_k$  
of weight $k$ as almost holomorphic functions $F(\tau,\bar \tau)$ on 
the upper-half plane which transforms like a modular form of weight $k$;
\[
F\big( \frac{a \tau + b}{c \tau + d}, 
\overline{ \frac{a \tau + b}{c \tau + d} } \big) = (c \tau + d)^k 
F(\tau,\bar \tau) \;\;.
\]
Then the ring of the almost holomorphic modular forms 
$\widehat M(\Gamma) =$ \break 
$\oplus_{k\geq 0} \widehat M(\Gamma)_k$ becomes a differential ring under 
$D_\tau:  \widehat M(\Gamma)_k \rightarrow \widehat M(\Gamma)_{k+2}$, with 
$D_\tau:=\frac{1}{2\pi i} \big( \frac{\pd \;}{\pd \tau} + 
\frac{k}{\tau-\bar\tau} \big)$. Elements of $\widehat M(\Gamma)$ have 
an expansion 
$F=\sum_{m\geq 0} c_m(\tau) \big(\frac{1}{\tau -\bar \tau}\big)^m$, 
and by taking the 
first coefficient $c_0(\tau)$ we obtain holomorphic objects. Kaneko-Zagier 
shows that this map defines a (differential) ring isomorphism 
$\varphi: \widehat M(\Gamma) \rightarrow \widetilde M(\Gamma)$, where 
$\widetilde M(\Gamma) = {\bf C}[E_2(\tau),E_4(\tau),E_6(\tau)]$ is the ring 
of the quasi-modular forms with the differential $\pd_\tau:=\frac{1}{2\pi i} 
\frac{\pd \;}{\pd \tau}$.  Our observation here is that the ring 
(\ref{eqn:KZ-0}) is defined by the K\"ahler geometry on the upper-half 
plane, and has a natural generalization to the Weil-Petersson geometry 
on the moduli space of Calabi-Yau manifolds. Based on this, we will introduce 
our BCOV ring ${\mathcal R}_{BCOV}^0$ in terms of purely geometric data, and 
subsequently introduce other forms of the ring, ${\mathcal R}_{BCOV}^\Gamma$, 
${\mathcal R}_{BCOV}^{hol}$. 
We may schematically write our parallelism of the BCOV rings to the relevant 
rings in Kaneko-Zagier theory:
\[
\begin{aligned} 
&{\bf C}[[\tau]] \big[ \frac{1}{\tau-\bar \tau} \big] & \supset \;\;\;
&\widehat M(\Gamma)  & \xrightarrow{\;\;\; \varphi \;\;\; } 
& \;\;\; \widetilde M(\Gamma) \\
&\;\;\;\;\; {\mathcal R}_{BCOV}^0 & \rightarrow \;\;\;
&{\mathcal R}_{BCOV}^{\Gamma}  &\xrightarrow{\bar t \rightarrow \infty} 
& \;\;\; {\mathcal R}_{BCOV}^{hol} \;\;.\\
\end{aligned}
\]
As one see in the arrow ${\mathcal R}_{BCOV}^0 \rightarrow 
{\mathcal R}_{BCOV}^{\Gamma}$, instead of $\supset$, the relations 
shown in this diagram are not exact correspondences but should be 
understood simply as parallelism. 
In fact, in our BCOV ring, following \cite{BCOV2}, we work with mermorphic 
sections of certain bundles instead of (almost) holomorphic forms in 
Kaneko-Zagier theory. Details will be described in the text, however 
it should be helpful to have this schematic diagram in mind. 

\vskip0.2cm 

The main result of this paper is the introduction of the BCOV ring 
${\mathcal R}_{BCOV}^0$ (Definition 3.3, Theorem 3.5) and making 
the parallelism to Kaneko-Zagier theory of quasi-modular forms complete.

\vskip0.2cm

Construction of this paper is as follows. To make the paper self-contained, 
in section 2, we review the geometry of the moduli space of Calabi-Yau 
manifold, which is called special K\"ahler geometry, and set up our 
notations. In section 3, we define our BCOV rings. In subsection (3-1), 
we introduce the first form of our BCOV (differential) ring 
${\mathcal R}_{BCOV}^0$ based on the special K\"ahler geometry 
(Definition 3.1, Theorem 3.5). 
We remark that the BCOV ring ${\mathcal R}_{BCOV}^0$ is 
infinitely generated, however there is a natural reduction 
${\mathcal R}_{BCOV}^{0,red}$ to a finitely generated ring. In (3-2), 
we consider modular (monodromy) property of the ring and 
we will define a differential ring ${\mathcal R}_{BCOV}^{\Gamma}$ as 
the ring of monodromy invariants. We, then, understand in our framework
the process of `fixing a holomorphic (mermorphic) ambiguities in 
the propagator functions' given in the section 6.3 of \cite{BCOV2}. 
In (3-3), BCOV rings ${\mathcal R}_{BCOV}^0$ 
and ${\mathcal R}_{BCOV}^{\Gamma}$ will be determined explicitly for an 
elliptic curve. There we provide a precise relation to the theory of 
quasi-modular forms (Propositions 3.11, 3.12). 
In (3-4), the final form ${\mathcal R}_{BCOV}^{hol}$ 
is defined under a choice of a symplectic basis of the middle dimensional 
homology group. In section 4, toward an application to the holomorphic anomaly 
equation, we introduce the holomorphic anomaly equation in the form 
appeared in \cite{YY},\cite{AL}. Considering holomorphic anomaly equation 
in the reduced ring ${\mathcal R}_{BCOV}^{0,red}$ we note that the equation 
simplifies to a system of linear differential equation (Proposition 4.3) 
which is easy to handle.  Conclusions and discussion are given in section 5. 
There, the equivalence of the modular anomaly equation in \cite{HST} to BCOV 
holomorphic anomaly equation is also announced.

\vskip0.5cm
\noindent
{\bf Acknowledgments:}
The main result of this paper was announced in the workshop 
``Number Theory and Physics at the Crossroads'', Sep. 21--26, 2008 at 
Banff International Research Station. The author would like to thank the 
organizers for providing a wonderful research environment there. He also 
would like to thank M.-H. Saito for valuable discussions. 
This work is supported in part by Grant-in Aid Scientific Research 
(C 18540014).

\vskip0.5cm


\section{{\bf Special K\"ahler geometry on deformation spaces} }


{\bf (2-1) Period integrals.} 
Let us consider a family of Calabi-Yau 3-folds 
${\mathcal Y}=\{ Y_x \}$ over a small complex domain $B$ with fibers 
$Y_x$ ($x \in B$). As such a family, we will consider hypersurfaces 
or complete intersections in projective toric varieties, and 
assume that the local family eventually extends to a family over a 
toric variety ${\mathcal M}$ $(={\bf P}_{Sec(\Sigma)}$ with the 
secondary fan, see eg. \cite{GKZ}). We also assume that 
${\rm dim} {\mathcal M} ={\rm dim} H^{2,1}(Y_{x_0})$ for a smooth $Y_{x_0}$. 

We fix a smooth $Y_{x_0} (x_0 \in B)$ as above. We denote the cohomology 
$H^3(Y_{x_0},{\mathbf Z})$ by $H_{x_0}$ and write the symplectic form there by 
\begin{equation}
\langle u, v \rangle := \sqrt{\text{--}1} \int_{Y_{x_0}} u \cup v \;\;.
\mylabel{eqn:symplectic}
\end{equation}
We choose a symplectic basis $\{ \alpha_I, \beta^J \}_{0 \leq I,J \leq r}$ 
satisfying $\langle \alpha_I, \beta^J \rangle = \delta_I^{\;J}$, 
$\langle \alpha_I, \alpha_J \rangle = \langle \beta^I, \beta^J \rangle=0$, 
and we denote its dual homology basis by 
$\{ A^I, B_J \}_{0 \leq I,J \leq r}$  
$(r:=\dim  H^{2,1}(Y_{x_0}))$. With respect to this basis, we write 
the symplectic form 
\[
Q =
\left( 
\begin{matrix} 0 & E \\ - E & 0 \end{matrix} \right)\;,
\]
where $E=E_{r+1}$ represents the unit matrix of size $(r+1)$.  
We define the period domain 
\[
{\mathcal D}=\{ [\omega] \in {\mathbf P}(H_{x_0} \otimes {\bf C}) \,| \, 
\langle \omega, \omega \rangle=0, \langle \omega, \overline \omega \rangle >0 \} \;.
\]
We choose a holomorphic three form of 
the fiber $Y_x$ $(x \in B)$ and denote it by $\Omega_x:=\Omega(Y_{x})$. 
Then the period map ${\mathcal P}_0: B \rightarrow {\mathcal D}$ is defined by
\[
{\mathcal P}_0(x) = 
\big[ \sum_I \left(\int_{A^I} \Omega_x \right)\alpha_I + 
\sum_J \left(\int_{B_J} \Omega_x \right) \beta^J\big] \;\;.
\]
Using path-dependent identification $H_3(Y_x,{\bf Z}) 
\cong H_3(Y_{x_0},{\bf Z})$, we may globalize this period map on $B$ to 
${\mathcal M}$ by introducing a covering space $\tilde {\mathcal M}$. 
We denote the resulting period map 
${\mathcal P}:\tilde {\mathcal M} \rightarrow {\mathcal D}$ and assume 
$\Gamma \subset Sp(2r+2,{\bf Z})$ as the covering group. 
In this paper, we will write the period map 
${\mathcal P}(x) = \big[ \vec\omega(x) \big]$ (or 
${\mathcal P}_0(x)=\big[ \vec\omega(x) \big] \; (x \in B)$) 
with the notations for the period integrals, 
\[
\vec\omega(x)=\sum_I X^I(x) \alpha_I + \sum_J P_J(x) \beta^J  \;\;.
\]

We write by ${\mathcal U}$ the restriction to ${\mathcal D}$ of the 
tautological line bundle ${\mathcal O}(-1)$ over 
${\bf P}(H_{x_0}\otimes {\bf C})$. We then set ${\mathcal L}
={\mathcal P}^* {\mathcal U}$, i.e., the pullback to 
$\tilde{\mathcal M}$. Complex conjugate of ${\mathcal L}$ will be 
denoted by $\overline{\mathcal L}$. 
The sections of ${\mathcal L}^{\otimes n} \otimes \overline 
{\mathcal L}^{\otimes m}$ will be often referred to 
as `sections of weight $(n,m)$'. 
The period integral $\vec\omega(x)$ may be considered as a section of 
${\mathcal L}$, and thus has weight $(1,0)$.

\vskip0.3cm

{\bf (2-2) Prepotential.} 
The symplectic form (\ref{eqn:symplectic}) naturally induces one form $\theta$ 
on ${\mathcal D}$ by $\theta := \langle d\omega, \omega \rangle$. With this 
one form, $({\mathcal D},\theta)$ becomes a holomorphic contact manifold of 
dimension $2r+1$. Since locally, the period map ${\mathcal P}_0 
: B \rightarrow {\mathcal D}$ is an embedding \cite{Bo}\cite{Ti}\cite{To} 
and also  
$\theta|_{{\mathcal P}_0(B)}=0$ due to Griffiths transversality, we 
know that the image of the period map ${\mathcal P}_0$ is a Legendre 
submanifold. Combining this with Gauss correspondence in projective geometry, 
it is found in general \cite{BG} that 
the image of the period map ${\mathcal P}_0$ can be recovered by the half 
of the period integrals $X^I(x)=\int_{A^I} \Omega_x$. More concretely, it is 
known that: 
\begin{list}{}{
\setlength{\leftmargin}{20pt}
\setlength{\labelwidth}{20pt}
}
\item[1)]
The map $x \mapsto [X^0(x), \cdots, X^r(x)] \in {\bf P}^r$ 
is an local isomorphism $B \rightarrow {\bf P}^r$, 
\item[2)]
Integrating 
$\theta|_{{\mathcal P}_0(B)}=0$ on $B$, we can write the other half of the 
period integrals,  
\[
P_J(x) = \frac{\partial {\mathcal F}(X^I)}{\partial X^J} \;\; (J=0,1,\cdots,r),
\]
in terms of a holomorphic function ${\mathcal F}(X)$  
called {\it prepotential}.  
\end{list}
The function ${\mathcal F}(X)$ is an holomorphic function  
of $X^0(x)$, $\cdots$, $X^r(x)$ and has the following homogeneous 
property 
\[
\sum_{I=0}^r X^I(x) \frac{\partial{\mathcal F}}{\partial X^I} = 2 
{\mathcal F(X)} \;\;.
\] 
This potential function exists locally for the small domain $B$. 
When globalizing 
the above local arguments to $\tilde {\mathcal M}$, we naturally see that the 
monodromy group $\Gamma$ plays a role for the definition of ${\mathcal F(X)}$ 
(see below). The group action of $\Gamma$ is referred to as 
{\it duality transformation} in physics(, see e.g. \cite{CRTP} and 
references therein). Here we remark that the holomorphic 
prepotential has a simple relation to the so-called Griffiths-Yukawa 
coupling of the family $\{ Y_x \}_{x \in B}$;
\begin{equation}
- \int_{Y_x} \hspace{-0.1cm} \Omega_x \cup  
\frac{\partial\;}{\partial x^i}
\frac{\partial\;}{\partial x^j}
\frac{\partial\;}{\partial x^k} \Omega_x = 
\hspace{-0.3cm} \sum_{I,J,K=0}^r 
\hspace{-0.2cm}
\frac{\partial X^I}{\partial x^i} 
\frac{\partial X^J}{\partial x^j} 
\frac{\partial X^K}{\partial x^k} 
\frac{\partial^3 {\mathcal F}(X)}{\partial X^I \partial X^J \partial X^K} \;,
\mylabel{eqn:yukawa-Cijk}
\end{equation}
where the l.h.s. is will be written by $C_{ijk}(x)$ hereafter. We denote 
$\overline C_{\bar i \bar j \bar k}$ the complex conjugate of $C_{ijk}(x)$.

\vskip0.3cm

{\bf (2-3) Period matrix (1) .} 
The most important object to introduce the BCOV anomaly equation is the 
classical period matrix. For simplicity, let us assume that our family 
$\{Y_x\}$ is given by a family of hypersurfaces in a (smooth) toric variety 
${\bf P}_\Sigma^4$, with the parameter $x=(x^1,x^2,\cdots,x^r)$ compactified 
to a toric variety ${\mathcal M}$. We then consider ${\mathcal A}^q_k$, the set of rational $q$ forms 
on ${\bf P}_\Sigma^4$ with a pole order less than 
$k$ along $Y_x$, and set the cohomology group 
\[
{\mathcal H}_k = {\mathcal A}^4_k/d {\mathcal A}^{3}_{k-1} \;\;.
\]
Obviously we have ${\mathcal H}_1 \subset {\mathcal H}_2 \subset \cdots $, 
and, in fact, this stabilizes at ${\mathcal H}_4$ $(={\mathcal H}_5=\cdots)$ 
to the rational 4 forms ${\mathcal H}$ of poles along $Y_x$. 
We note that, for 3-folds, the `tubular' map 
$\tau: H_3(Y_x,{\bf Z}) \simrightarrow 
H_4({\bf P}_\Sigma \setminus Y_x, {\bf Z})$ is an isomorphism
([Gr, Proposition 3.5]), and is related to the Poincar\'e 
residue map $R: {\mathcal H} \rightarrow H^3(Y_x,{\bf C})$ along $Y_x$ by 
\[
\int_{\tau(\gamma)} \omega = \int_{\gamma} R(\omega) \;\;(\omega \in {\mathcal H}_k).
\]
This residue map is an isomorphism, and more precisely, maps the 
filtration 
${\mathcal H}_1 \subset {\mathcal H}_2 \subset {\mathcal H}_3 \subset 
{\mathcal H}_4$ to the Hodge filtration 
\[
F^{3,0} \subset F^{3,1} \subset F^{3,2} \subset F^{3,3} \;\;.
\mylabel{eqn:Hodge-filt}
\]
The holomorphic three form $\Omega_x=R(\omega_0)$ may then be given by a basis 
$\omega_0$ of ${\mathcal H}_1 \cong F^{3,0}$. We take a basis 
$\omega_0, \omega_1, \cdots, \omega_r$ of ${\mathcal H}_2$, and define 
the period matrix 
\begin{equation}
{\bf \Omega}=\left( 
\begin{matrix} 
\int_{\tau(A_0)} \omega_0 & \cdots &
\int_{\tau(A_r)} \omega_0 & 
\int_{\tau(B_0)} \omega_0 & \cdots &
\int_{\tau(B_r)} \omega_0 \\
& \vdots &  & & \vdots \cr
\int_{\tau(A_0)} \omega_r & \cdots &
\int_{\tau(A_r)} \omega_r & 
\int_{\tau(B_0)} \omega_r & \cdots &
\int_{\tau(B_r)} \omega_r \\
\end{matrix} 
\right)
\mylabel{eqn:period-matrix}
\end{equation}
as $(r+1,2r+2)$ matrix. The first row of this matrix coincides with the 
period integral $\vec\omega(x)=\sum_I X^I(x) \alpha_I + \sum_J P_J(x) 
\beta^J$, and the monodromy group $\Gamma$ acts on this period matrix from 
the right. 

The following properties of ${\bf \Omega}$ are consequences of the filtration 
(\ref{eqn:Hodge-filt}) and the Hodge-Riemann bilinear relations;
\begin{equation}
\begin{matrix}
& 1)& \;\; {\bf \Omega} \; Q \; {}^t\!{\bf \Omega}=0 \;\; \hfil & \hspace{5cm}\\
& 2)& \;\; \sqrt{\text{--}1} 
           {\bf \Omega} \; Q \; {}^t\!\overline{\bf \Omega} >0 \;\;. \hfil &\;\; 
\end{matrix}
\mylabel{eqn:bilinear-rel}
\end{equation}
Here 2) means that when we decompose the $(r+1) \times (r+1)$ hermitian 
matrix $\sqrt{-1}{\bf \Omega} Q {}^t\!\overline{\bf \Omega}$ 
into the block form 
compatible with the filtration ${\mathcal H}_1 \subset {\mathcal H}_2$, 
then the first diagonal block ( $1\times 1$ matrix) is positive definite 
and the whole $(r+1) \times (r+1)$ hermitian matrix has $1$ 
positive eigenvalue and $r$ negative eigenvalues. 

\vskip0.1cm

In our case of hypersurfaces (or complete intersections) in toric varieties 
${\bf P}_\Sigma^4$, the basis $\omega_0$ for the rational differential 
${\mathcal H}_1$ can be given explicitly by the defining equation of $Y_x$ 
with the deformations $x^1,\cdots,x^r$. Then our assumption hereafter 
for the family $\{ Y_x \}_{x \in B}$ is that the derivatives 
\begin{equation}
\partial_1 \omega_0 , \cdots,
\partial_r \omega_0
\quad (\partial_i \omega_0 :=\frac{\partial \;}{\partial x^i} \omega_0) 
\mylabel{eqn:basis-H2}
\end{equation}
span ${\mathcal H}_2$ together with the basis $\omega_0$ of ${\mathcal H}_1$ 
for each small complex domain $B \subset {\mathcal M}$.  
It will be useful to define a notation $\partial_0 \omega_0 \equiv \omega_0$. 
Now using $\int_{\tau(\gamma)} \partial_i \omega_0 = \int_\gamma 
R( \partial_i \omega_0) = \partial_i \int_\gamma R(\omega_0)$, we can write 
the period matrix simply by 
\begin{equation}
{\bf \Omega} = 
\left( \begin{matrix} 
\partial_i X^I & \partial_i P_J  
\end{matrix} \right) = 
\left( \begin{matrix} 
\partial_i X^I &  \partial_i X^J \tau_{IJ}  
\end{matrix} \right)_{0 \leq i,I \leq r}
\mylabel{eqn:period-matrix-XP}
\end{equation}
where we define $\tau_{IJ}=\frac{\partial^2\;}
{\partial X^I \partial X^J} {\mathcal F}(X)$, 
and set $\partial_0 X^I \equiv X^I, \partial_0 P_J \equiv P_J$. 
In the above formula, and also hereafter, the repeated indices are assumed to 
be summed over (Einstein's convention) unless otherwise mentioned. 
Using the property 1) in the previous paragraph (2-2), we may assume 
\[
\det ( \partial_i X^I )_{0 \leq i,I \leq r} = (X^0)^{r+1} 
\det \left( \partial_i \Big(\frac{X^I}{X^0}\Big) \right)_{1 \leq i,I \leq r}\not=0
\]
for $x$ such that $X^0(x)\not=0$. Hence, at least locally, we can 
normalize the period matrix in the form 
\[
(\partial_i X^I)^{-1} {\bf \Omega} =  ( E \; \tau  ) \;\;.
\]
In this normalized form, the bilinear relation 1) in (\ref{eqn:bilinear-rel}) 
is trivial since $\tau_{IJ}=\tau_{JI}$, while 2) entails 
\[
\sqrt{\text{--}1} ( \overline \tau -\tau ) >0 \;\;,
\]
which means the matrix ${\rm Im} \, \tau$ has one  positive and 
$r$ negative eigenvalues. Here we note a similarity  
to the period matrix of genus $g$ curves, however the mixed property of the 
eigenvalues is a new feature in higher dimensions. 

Finally, we note that the monodromy group $\Gamma$ acts on the 
normalized period matrix from the right, and for $\left( 
\begin{smallmatrix} D & B \\ C & A 
\end{smallmatrix} \right) \in \Gamma$ we have 
\begin{equation}
\tau \mapsto (C\tau+D)^{-1} (A\tau+B) \;\;.
\mylabel{eqn:ABCD-act-tau}
\end{equation}
Since $\frac{\pd \;}{\pd X^I} \frac{\pd\;}{\pd X^J} {\mathcal F}=\tau_{IJ}$, 
this describes the transformation 
property of the prepotential which is defined locally for the family 
over $B$.

\vskip0.3cm

{\bf (2-4) Period matrix (2).} 
Period matrix (\ref{eqn:period-matrix-XP}) has been defined entirely 
in the holomorphic category, since it is based on the Hodge filtration.  
One may modify the Hodge filtration to Hodge decomposition if we incorporate 
a hermitian structure coming from the K\"ahler geometry on ${\mathcal M}$. 
Let us first note that over the small domain $B$, and hence on ${\mathcal 
M}$ except the degeneration loci, there exists 
a K\"ahler metric $g_{i\bar j}=\partial_i \partial_{\bar j} 
K(x,\bar x)$ $(1 \leq i,\bar j \leq r)$ with the K\"ahler potential 
\[
K(x,\bar x)
=- \log \{ \langle \Omega_x, \overline \Omega_x \rangle \}
=- \log \{ \langle R(\omega_0), \overline{ R(\omega_0)} \rangle \}
\;\;.
\]
This K\"ahler metric is called Weil-Petersson metric on ${\mathcal M}$.
The bases $\partial_0 \omega_0\equiv \omega_0; \partial_1 \omega_0, \cdots, 
\partial_r \omega_0$ which are compatible with the filtration ${\mathcal H}_1 
\subset {\mathcal H}_2$, or the Hodge filtration $F^{3,0} \subset F^{3,1}$, 
may now be modified to 
\[
D_0 \omega_0\equiv \omega_0; D_1 \omega_0, \cdots, 
D_r \omega_0 \;\; (D_i = \partial_i + K_i, i=1,\cdots,r),
\]
where $K_i = \partial_i K(x,\bar x)$. One should observe that, since $\langle 
R(\omega_0), R(D_i \omega_0) \rangle$ $=0$ holds for $i=1,\cdots,r$, 
these bases are compatible to the Hodge 
decomposition $H^{3,0}(Y_x) 
\oplus H^{2,1}(Y_x)$. Correspondingly, the period matrix 
(\ref{eqn:period-matrix-XP}) may be modified to 
\begin{equation}
{\bf \Omega} = \left( \,
D_i X^I \;  D_i P_J \, \right) =
\left( \,
D_i X^I \; D_i X^J \tau_{IJ} \, \right)_{0 \leq i,I \leq r} \;\;,
\mylabel{eqn:period-matrix-D}
\end{equation}
where we set $D_0 X^I \equiv X^I, D_0 P_J \equiv P_J$. 
Since we have $\det (D_i X^I)_{0 \leq i,I \leq r}$ 
$= \det ( \partial_i X^I )_{0 \leq i,I \leq r} \not=0$, the normalized period 
integral has a similar form as before;
\[
(D_i X^I)^{-1} {\bf \Omega} = \left( \begin{matrix} \; E &\tau \; 
\end{matrix} \right)
\;\;. 
\]

The same monodromy group $\Gamma$ acts from the right on 
the period matrix.  In contrast to this, the left action shows 
a nice connection to the K\"ahler geometry on ${\mathcal M}$.

\vskip0.3cm

{\bf (2-5) Special K\"ahler geometry on ${\mathcal M}$.} 
After some algebra, we have for the K\"ahler metric 
\[
g_{i\bar j}=\partial_i \partial_{\bar j} K(x,\bar x) = 
-e^{K(x,\bar x)} \langle R(D_i \omega_0), \overline{R(D_j \omega_0)} \rangle .
\]
With respect to this, we introduce the metric connections 
\[
\Gamma_{i j}^k = g^{k\bar k} \partial_i g_{j \bar k} \;\;,\;\;
\Gamma_{\bar i \bar j}^{\bar k} = g^{\bar k k} \partial_{\bar i} g_{\bar j k} \;,
\]
for the holomorphic tangent bundle $T{\mathcal M}$ and the anti-holomorphic 
tangent bundle $T'{\mathcal M}$, respectively. The curvature tensor for these 
connections are given by 
\[
R_{i \bar j \; l}^{\;\; k} = - \partial_{\bar j} \Gamma_{il}^{k}\;\;,\;\;
R_{i \bar j \; \bar l}^{\;\; \bar k} = 
\partial_{i} \Gamma_{\bar j \bar l}^{\bar k}\;\;.
\]
In addition to these, we have the Griffiths-Yukawa couplings $C_{ijk}(x)$ and 
$\overline C_{\bar i \bar j \bar k}(\bar x)$ on the moduli 
space ${\mathcal M}$. These 
tensors define the so-called {\it special K\"ahler geometry} on 
${\mathcal M}$, which can be 
summarized into a property of the period matrix (\ref{eqn:period-matrix-D}).

Let us introduce the following $(2r+2)\times(2r+2)$ matrix
\begin{equation}
\left( 
\begin{matrix} {\bf \Omega} \\ \overline{\bf \Omega} \\ \end{matrix} 
\right) 
=\left( 
\begin{matrix} D_iX^I & D_i P_J  \\ 
\overline{D_i X^I} & \overline{D_i P_J} 
\\ \end{matrix} 
\right)_{0 \leq i,I,J \leq r}
\;\;.
\mylabel{eqn:ObarO}
\end{equation}
By definition of the period matrix, the row vectors represent 
the Hodge decomposition $H^{3,0}(Y_x) \oplus H^{2,1}(Y_x) \oplus 
H^{1,2}(Y_x) \oplus H^{0,3}(Y_x)$ in terms of the bases
\begin{equation}
R(\omega_0), \;\; 
R(D_i \omega_0) \;, \;\;
\overline{R(D_i \omega_0)} \;,\;\;
\overline{R(\omega_0)} \;\;. 
\mylabel{eqn:Hodge-basis}
\end{equation}
By simply writing (\ref{eqn:ObarO}), we implicitly understand 
the row vectors are ordered according to the Hodge decomposition above. 
With these explicit bases in mind, we introduce the covariant derivatives 
acting on the column vectors ${\bf \Omega}$ and $\overline{\bf \Omega}$;
\[
D_i=\begin{cases} 
\partial_i + K_i - \Gamma_{i*}^{*} & \text{on } {\bf \Omega} \\
\partial_i & \text{on } \overline{\bf \Omega} \end{cases}
\;,\;\;
\overline D_{\bar i}=\begin{cases} 
\partial_{\bar i} & \text{on } {\bf \Omega} \\ 
\partial_{\bar i} + K_{\bar i} - \Gamma_{{\bar i}*}^{*} & \text{on } 
\overline{\bf \Omega} \end{cases} ,
\]
where $\Gamma_{i*}^{*}$ and $\Gamma_{\bar i *}^{*}$ represents the 
conventional form of the contraction via the metric connection. 

\vskip0.3cm

\begin{theorem} \thlabel{th:flat-A} 
The period matrix satisfies 
\begin{equation}
(D_i+{\mathcal A}_i) 
\left( 
\begin{matrix} {\bf \Omega} \\ \overline{\bf \Omega} \\ \end{matrix} 
\right) =
\left(
\begin{matrix} {\bf 0} \\ {\bf 0} \\ \end{matrix} 
\right)
 \;\;,\;\;
(\overline D_{\bar i}+\bar{\mathcal A}_{\bar i}) 
\left( 
\begin{matrix} {\bf \Omega} \\ \bar{\bf \Omega} \\ \end{matrix} 
\right) =
\left(
\begin{matrix} {\bf 0} \\ {\bf 0} \\ \end{matrix} 
\right)
 \;\;,\;\;
\mylabel{eqn:spK1}
\end{equation}
with 
\[
{\mathcal A}_i=
\bordermatrix{
   & _0 & _n & _{\bar n} & _{\bar 0} \cr
\,_0 & 0 & - \delta_i^n & 0 & 0 \cr
_m & 0 & 0 & {\mathcal C}_{im}^{\bar n} & 0 \cr
_{\bar m} & 0 & 0 & 0 & -g_{i\bar m} \cr
\,_{\bar 0} &0 & 0 & 0 & 0 \cr },\; 
\bar{\mathcal A}_{\bar i}=
\bordermatrix{
&  _0 & _{n} & _{\bar n} & _{\bar 0} \cr
\,_{0}      &0 & 0 & 0 & 0 \cr
_{m}      &-g_{\bar i m} & 0 & 0 & 0 \cr
_{\bar m}  &0 & \overline{\mathcal C}_{\bar i \bar m}^{n} & 0 & 0 \cr
\,_{\bar 0}  &0 & 0 & -\delta_{\bar i}^{\bar n} & 0 \cr},
\]
where we set ${\mathcal C}_{im}^{\bar n}=
\sqrt{\text{--}1}e^{K} C_{imn}g^{n\bar n}$ and 
$\overline{\mathcal C}_{\bar i \bar m}^{n} = 
-\sqrt{\text{--}1}e^{K} \overline C_{\bar i \bar m \bar n}g^{n \bar n}$.
\end{theorem}

\begin{proof} Let us consider, the three from $R(D_j \omega_0)$ 
in the Hodge decomposition (\ref{eqn:Hodge-basis}). 
Then $D_i R(D_j \omega_0)$ is a three form and may be 
expressed in terms of the basis (\ref{eqn:Hodge-basis}) as 
\[
D_i R(D_j \omega_0) = 
c_0 R(\omega_0)+c_m R(D_m \omega_0) + 
d_m \overline{R(D_m \omega_0)} + d_0 \overline{R(\omega_0)} \;.
\] 
Now, using $\langle R(D_n \omega_0), \overline{R(D_m \omega_0)} 
\rangle = -e^{-K}g_{n \bar m}$ and 
other orthogonal relations, it is easy to see $c_0=c_m=d_0=0$ and  
\[
\begin{aligned}
- d_m e^{-K} g_{k \bar m} 
& =\langle R(D_k \omega_0), D_i R(D_j \omega_0) \rangle  \\
& = - \langle R(\omega_0), D_k D_i D_j R(\omega_0) \rangle 
= \sqrt{\text{--}1} C_{ijk} \;\;, 
\end{aligned}
\]
where $C_{ijk}$ is the Griffiths-Yukawa coupling 
(\ref{eqn:yukawa-Cijk}) ($R(\omega_0)= \Omega_x$).
One can continue similar arguments for other bases of the 
Hodge decomposition (\ref{eqn:Hodge-basis}). Integrating 
three forms over the cycles $\{ A^I, B_J \}$, we obtain the claimed 
linear relations for the row vectors of the 
period matrix. 
\end{proof} 


The connection matrix ${\mathcal A}_i, {\mathcal A}_{\bar i}$ was first  
determined by Strominger in \cite{St1}. 
The existence of the first order differential operator is due to the 
fact that the Hodge decomposition is `flat' over ${\mathcal M}$, 
and we should have 
the compatibility relations for the first order system. 
Indeed one can see that  
\[
[ D_i + {\mathcal A}_i, D_j + {\mathcal A}_j ]=0 =  
[ \overline D_{\bar i} + \overline{\mathcal A}_{\bar i}, \overline D_{\bar j} + 
\overline{\mathcal A}_{\bar j} ], 
\]
are ensured by the existence of the prepotential (\ref{eqn:yukawa-Cijk}), and 
another mixed-type compatibility condition imposes a rather strong 
constraint on the K\"ahler geometry\cite{St1}, which is called 
special K\"ahler geometry(, see, e.g., \cite{CRTP} and references therein). 

\vskip0.3cm

\begin{theorem} \thlabel{th:flat-R}
The compatibility condition 
\[
[ D_i + {\mathcal A}_i, \overline D_{\bar j} + 
\overline{\mathcal A}_{\bar j} ] = 0
\]
is equivalent to 
\begin{equation}
- R_{i \bar j \; l}^{\;\;k} = g_{i\bar j}\delta_l^{\;k} + 
g_{\bar j l}\delta_i^{\;k} 
- e^{2K}C_{ilm}\overline C_{\bar j \bar k \bar m} g^{m \bar m}g^{k \bar k} \;\;.
\mylabel{eqn:spK2}
\end{equation}
\end{theorem}

\vskip0.1cm

In section (3-1), we will derive the relation (\ref{eqn:spK2}) directly 
evaluating 
the metric connection. The both equations (\ref{eqn:spK1}) and 
(\ref{eqn:spK2}) are often referred to as 
special K\"ahler geometry relations, and will play central roles in 
solving BCOV anomaly equation.

\vskip0.3cm

{\bf (2-6) Notations. } As we have summarized above, the special K\"ahler 
geometry relations, 
in the holomorphic local coordinate $x^i (i=1,\cdots,r)$, on 
${\mathcal M}$ arises from the flat property of the 
period matrix ${\bf \Omega}$ and its complex conjugate. Since the row 
vectors of ${\bf \Omega}$ 
correspond to the decomposition $H^{3,0}(Y_x) \oplus H^{2,1}(Y_x)$, 
it is convenient to introduce the (Greek letter) notation 
$\alpha = (0, i)$ and $\bar\alpha =(\bar 0,\bar i)$ for the indices of 
the row vectors and their complex conjugates, respectively. 
Then the period matrix may be written simply by 
\[
{\bf \Omega} = ( D_\alpha X^I\, D_\alpha P_J ) = 
(D_\alpha X^I) \; ( \, E \;\; \tau_{IJ} \, ) \;\;.
\]
As remarked in subsection (2-4), $(r+1) \times (r+1)$ matrix $(D_\alpha X^I)$ 
is invertible. We set 
$\xi_{\alpha}^{\; I} := D_\alpha X^I$,  
$\bar\xi_{\bar\alpha}^{\; I} := \overline{D_\alpha X^I}$ 
and define 
$\xi_I^{\;\alpha}$ and  
$\bar\xi_I^{\;\bar\alpha}$, respectively,  
by 
\[
( \xi_I^{\;\alpha} ) = (D_\alpha X^I)^{-1} \;\;,\;\;
( \bar\xi_I^{\;\bar\alpha} ) = (\overline{D_\alpha X^I})^{-1} \;\;.
\]


Note that the $i$-th row of ${\bf \Omega}$ represents period integrals of 
the three form $R(D_i \omega_0) \in H^{2,1}(Y_x)$ for a symplectic basis 
$\{A^I, B_J\}$. Then we can see the Weil-Petersson metric in the following 
$(r+1)\times (r+1)$ hermitian matrix,
\[
\left( 
\begin{matrix} g_{0 \bar 0} & 0 \\
0 & g_{i \bar j} \\
\end{matrix} \right) 
= - e^{K(x,\bar x)} \sqrt{\text{--}1} {\bf \Omega} Q {}^t\!\overline{\bf \Omega} \quad
(g_{0\bar 0}\equiv -1).
\]
Equivalently one can write this matrix relation by 
\[
g_{\alpha \bar\beta}=
\sqrt{\text{--}1}e^{K(x,\bar x)} \; 
\xi_{\alpha}^{\;I} \;
(\tau-\bar\tau)_{IJ} \;
\bar\xi_{\bar\beta}^{\;J} \;\;.
\] 
For the tensor analysis in later sections, we introduce a ``metric'' by  
\[
({\mathcal G}_{IJ})= \sqrt{\text{--}1}e^{K(x,\bar x)} 
\big( \tau-\overline\tau \big) \;\;,\;\;
({\mathcal G}^{IJ})= -\sqrt{\text{--}1}e^{-K(x,\bar x)} 
\big( \tau-\overline\tau \big)^{-1} \;. 
\]
Then we have $g_{\alpha\bar\beta}=\xi_\alpha^{\;I}{\mathcal G}_{IJ}
\bar\xi_{\bar\beta}^{\;J}, 
g^{\alpha\bar\beta}=\xi_I^{\;\alpha}{\mathcal G}^{IJ}
\bar\xi_J^{\;\bar\beta}$.  With these metrics we will raise and lower the 
indices $I,J,\cdots$ as well as the Greek indices. 

\vskip0.1cm

Now the special K\"ahler geometry relations in Theorem\ref{th:flat-A} 
may be expressed by
\begin{equation}
\begin{cases} 
D_i \xi_0^{\;I}= \xi_i^{\;I} & \\
D_i \xi_j^{\;I}=-\sqrt{\text{--}1} 
e^{K} C_{ijk} g^{k\bar k}\bar\xi_{\bar k}^{\;I} & \\
D_i \bar \xi_{\bar j}^{\;I}=g_{i\bar j} \bar \xi_{\bar 0}^{\;I} & \\
D_i \bar\xi_{\bar 0}^{\;I}=0 & \\
\end{cases} \hspace{-0.2cm}
\begin{cases} 
\overline D_{\bar i} \xi_{0}^{\;I}= 0 \\
\overline D_{\bar i} \xi_{ j}^{\;I}= g_{\bar i j} \xi_{0}^{\;I} \\
\overline D_{\bar i} \bar\xi_{\bar j}^{\; I} =
\sqrt{\text{--}1} e^{K} 
\overline C_{\bar i \bar j \bar k} g^{k\bar k} \xi_{k}^{\;I} & \\
\overline D_{\bar i} \bar \xi_{\bar 0}^{\;I}= \bar \xi_{\bar i}^{\;I} & 
\end{cases} \hspace{-0.2cm}.
\mylabel{eqn:special-K-tensor}
\end{equation}
It will be useful to note that the following relation holds by definition; 
\begin{equation}
D_i \xi_j^{\;J}=C_{ijm} S^{mn} \xi_n^{\;I}\;\;,
\mylabel{eqn:Dxi-CS}
\end{equation}
where $S^{mn}$ is the {\it propagator} that will be introduced in 
the next section.

\vskip1cm

\section{{\bf BCOV rings}}

{\bf (3-1) BCOV ring ${\mathcal R}^0_{BCOV}$. } 
Based on the special K\"ahler geometry relations summarized in the previous 
section, we introduce a differential ring over the mermorphic sections of 
a certain vector bundle over $\tilde{\mathcal M}$. 
Let us first introduce the so-called propagators;

\vskip0.1cm
\begin{definition} \mylabel{def:Sab}
\[
S_{\bar\alpha\bar\beta}=e^{2K}(\tau-\overline\tau)_{IJ}
\bar\xi_{\bar\alpha}^{\;I}
\bar\xi_{\bar\beta}^{\;J}
\;\;,\;\;
S^{\alpha\beta}=g^{\alpha\bar\alpha}g^{\beta\bar\beta}S_{\bar\alpha\bar\beta} \;\;.
\]
\end{definition}

\vskip0.1cm

Obviously $S_{\bar\alpha\bar\beta}$ and $S^{\alpha\beta}$ are symmetric 
with respect to the indices. Note that, since the factor $e^{2K}$ 
has weight $(-2,-2)$, both $S^{\alpha\beta}$ 
and $S_{\bar\alpha \bar\beta}$ have weight $(-2,0)$ with respect to the  
line bundle ${\mathcal L}$ (, see section (2-1)).  
Following \cite{BCOV2}, we will often use the notation $S, S^i, S^{ij}$, 
which are related to $S^{\alpha\beta}$ by
\[
\left( \begin{matrix} S^{00} & S^{0i} \\ S^{0i} & S^{ij} \\ \end{matrix} \right)
=
\left( \begin{matrix} 2 S & -S^{i} \\ -S^{i} & S^{ij} \\ \end{matrix} \right) \;\;.
\]
If we use the property of $\tau$ and the prepotential introduced in (2-2), 
we have  
\[
S=\frac{1}{2}e^{2K}(\tau-\overline\tau)_{IJ}\bar X^I\bar X^J
=e^{2K}\Big( \frac{1}{2} 
         {\tau}_{IJ}\bar X^I \bar X^J - \overline{\mathcal F}(\bar X) \Big) \;\;.
\]

\vskip0.1cm

\begin{proposition} \mylabel{prop:DS}
The following relations hold, 
\[
1) \quad \overline D_{\bar i} S = S_{\bar0\bar i}  \qquad 
2) \quad \overline D_{\bar i} \overline D_{\bar j} S =S_{\bar i \bar j}  \qquad
3) \quad \overline D_{\bar i} \overline D_{\bar j} \overline D_{\bar k} S = 
e^{2K} \overline C_{\bar i \bar j \bar k} \;\;.
\]
These are equivalent to 
{\rm 1)} $\pd_{\bar i} S=g_{\bar i j}S^j$,  
{\rm 2)} $\pd_{\bar i} S^{j} = g_{\bar i i}S^{ij}$ and  
{\rm 3)} $\pd_{\bar i}S^{jk}=e^{2K}
\overline C_{\bar i \bar j \bar k}g^{j\bar j} g^{k \bar k}$. 
\end{proposition}

\begin{proof}
Acting $D_{\bar i}(=\pd_{\bar i})$ on 
$S=\frac{1}{2}e^{2K}(\tau-\bar\tau)_{IJ} \bar X^I \bar X^J$, we obtain 
\[
\pd_{\bar i}S=e^{2K}(\tau-\overline\tau)_{IJ}D_{\bar i}\bar X^I \bar X^J 
= S_{\bar 0 \bar i} \;\;,
\]
where we use 
$(\frac{\pd \;}{\pd \bar X^K} \overline \tau_{IJ}) \bar X^J =  \bar X^J 
(\frac{\pd \;}{\pd \bar X^J} \overline \tau_{IK}) =0$. 
By similar calculations with the special K\"ahler geometry relation 
(\ref{eqn:special-K-tensor}), the properties 2) and 3) above follow. 
\end{proof}

\vskip0.1cm

$S^{kl},S^k,S$ are called propagators in physics literatures, and 
contain the anti-holomorphic prepotential $\bar{\mathcal F}(\bar X)$ 
in their definitions. On the other hand, 
the holomorphic prepotential ${\mathcal F}(X)$ 
defines the so-called $n$-point functions on a Riemann sphere by 
\[
C_{i_1 i_2 \cdots i_n} = D_{i_1}D_{i_2}\cdots D_{i_n} {\mathcal F}(X) \quad (n\geq 3), 
\]
where the covariant derivative 
$D_i=\pd_i + 2 \pd_i K(x,\bar x) - \Gamma_{i *}^{*}$ acts on ${\mathcal F}(X)$ 
by the standard contractions of the holomorphic indices 
(see (\ref{eqn:D-tensor}) in general). 
By K\"ahler geometry, the holomorphic covariant derivatives commute 
with each other, so the $n$-point functions are symmetric tensor, 
and in particular $C_{ijk}$ coincides with the 
Griffith-Yukawa coupling (\ref{eqn:yukawa-Cijk}).

\vskip0.1cm

\begin{definition} 
As a symmetric algebra, we define 
\[ 
{\mathcal R}^0_{BCOV}={\bf Q}[S, S^i, S^{ij}, K_{i_1},K_{i_1i_2},\cdots, C_{i_1i_2i_3}, 
\cdots, C_{i_1i_2i_3\cdots i_n},\cdots ] \;,
\]
where $K_i:=D_iK(x,\bar x)=\frac{\pd \;}{\pd x^i} K(x,\bar x)$, 
$K_{i_1i_2}:=D_{i_1}D_{i_2}K(x,\bar x)$ and so on, We call 
this symmetric algebra {\it BCOV ring}. 
\end{definition}

\vskip0.1cm

In general this BCOV ring is infinitely generated. However we will see 
in the next section that if we 
consider the corresponding ring for an elliptic curve, the ring is 
finitely generated. In case of Calabi-Yau 
threefolds, as it turns out later that this problem of finiteness is 
related to the explicit form 
of the holomorphic function $h_{ijkl}(x)$ 
derived in (\ref{eqn:hijkl}) below. 
For convenience, we will often abbreviate the infinite 
series of the generators $K_i, D_iK_j,D_iD_jK_k\cdots$ by $\{K_i\}$ and 
similarly for $\{ C_{ijk} \}$. 
With this convention, 
the BCOV ring may be written simply by 
\begin{equation}
{\mathcal R}_{BCOV}^0={\bf Q}[S,S^i,S^{ij},\{ K_i \}, \{ C_{ijk} \} ] \;\;.
\mylabel{eqn:abbrev-gen}
\end{equation}

The propagator $S^{\alpha\beta}$ and $C_{i_1i_2\cdots i_n}$ have 
respective weights $(-2,0)$ and $(2,0)$, and also $K_i$ has weight $(0,0)$. 
Therefore 
the symmetric algebra is graded and defined in the set of global sections of 
\begin{equation}
\bigoplus_{m,n \geq 0} \bigoplus_{k=-\infty}^{\infty}
\big( \pi^*(T^*{\mathcal M}) \big)^{\otimes m} \otimes 
\big( \pi^*(T{\mathcal M}) \big)^{\otimes n} \otimes {\mathcal L}^k \;\;,
\mylabel{eqn:sym-tensor}
\end{equation}
where $\pi: \tilde{\mathcal M} \rightarrow {\mathcal M}$ is the covering map 
and ${\mathcal L} \rightarrow \tilde{\mathcal M}$ is the line bundle 
introduced in (2-1). Precisely, $K_i$ is a connection of the holomorphic 
line bundle ${\mathcal L}$ and 
therefore this is not a one form. However we regard $K_i$ as a one form 
taking a global trivialization of ${\mathcal L}$ over $\tilde {\mathcal M}$. 
By the following 
lemma, we see how  
the BCOV ring ${\mathcal R}^0_{BCOV}$ depends on the sheets of 
the covering.

\vskip0.1cm
\begin{lemma} When we change the symplectic basis by
${\bf \Omega}\rightarrow{\bf \Omega}
\big( \hspace{-1pt} 
\begin{smallmatrix} D & B \\ C & A \end{smallmatrix} 
\hspace{-1pt}\big)$, 
we have the corresponding change of the generators,
\begin{equation} 
S^{\alpha\beta} \rightarrow S^{\alpha\beta} + 
\xi_I^{\; \alpha} \big[ C (D+\tau C) \big]^{IJ} \xi_J^{\; \beta} \;\;.
\mylabel{eqn:propagator-Tr}
\end{equation}
\end{lemma}


\begin{proof} We keep our convention of the contraction by writing 
the indices of the symplectic matrix by 
$(X^I P_J) \rightarrow (X^I P_J)  
\left( \begin{smallmatrix} D_{I}^{\;J} & B_{IJ} \\ C^{IJ} & A^{I}_{\;J} \\ 
\end{smallmatrix}\right)$. Then we have 
\[
\xi_{\alpha}^{\; I} \rightarrow \xi_\alpha^{\; J}(D+\tau C)_{J}^{\;\; I} \;\;,\;\; 
\tau_{IJ} \rightarrow \big[ (D+\tau C)^{-1} (B+\tau A) \big]_{IJ} \;\;,
\]
and also, using 
$ \left( \begin{smallmatrix} D & B \\ C & A \\ 
\end{smallmatrix} \right) Q \;{}^t\!
\left( \begin{smallmatrix} D & B \\ C & A \\ 
\end{smallmatrix} \right) = Q$, we have
\[
(\tau -\bar\tau) \rightarrow (D+\tau C)^{-1} \, (\tau-\bar\tau) \, 
{}^t\!\overline{(D+\tau C)}^{-1} \;\;.
\]
After some algebra, the claimed transformation property follows directly 
from the definitions.
\end{proof}

\vskip0.1cm

We note that the metric connection $\Gamma_{ij}^k$ and $K_i$ define 
the covariant derivative $D_i$ 
on the sections (\ref{eqn:sym-tensor}). Thus, for example, for 
the section $V_{i}^{\;jl}$ of weight $(k,0)$ we have 
\begin{equation}
D_n V_{i}^{\;jl} = \frac{\pd \;}{\pd x^n} V_{i}^{\;jl} 
- \Gamma_{ni}^m V_{m}^{\;jl}
+ \Gamma_{nm}^j V_{i}^{\;ml} 
+ \Gamma_{nm}^l V_{i}^{\;jm} 
+ k K_n \,V_{i}^{\;jl} \;\;.
\mylabel{eqn:D-tensor}
\end{equation}

\vskip0.1cm

\begin{theorem} The BCOV ring ${\mathcal R}_{\rm BCOV}^0$ is a graded, 
differential, symmetric algebra 
with the (commuting) differentials $D_i \; (i=1,\cdots,r)$.
\end{theorem}

\vskip0.1cm
This theorem is a direct consequence of the following proposition.
\vskip0.1cm

\begin{proposition} \mylabel{prop:DS-DK}
The covariant derivative acts on the generators of 
${\mathcal R}_{\rm BCOV}^0$ by
\begin{equation}
\begin{aligned}
D_i S^{kl}&=\delta^k_i S^l+\delta^l_i S^k-C_{imn} S^{mk}S^{nl} \;,\\
D_iS^k\; &=-C_{imn}S^m S^{nk} + 2 \delta^k_i S \;,\\
D_i S\;\; &=-\frac{1}{2}C_{imn}S^m S^n \;,\\
D_i K_j &= -K_i K_j +C_{ijm}S^{mn}K_n-C_{ijm}S^m +h_{ij}\;,
\end{aligned}
\mylabel{eqn:DS-DK}
\end{equation}
where $h_{ij}$ is a holomorphic function.
\end{proposition}


\begin{proof} 
The first three equations follow from the definitions and 
the special K\"ahler 
geometry relations (\ref{eqn:spK2}), (\ref{eqn:special-K-tensor}). 
There, it is useful to write $S^{\alpha\beta}=\frac{1}{\sqrt{\text{--}1}} 
e^K {\mathcal G}^{IJ} \xi_I^{\;\alpha} \xi_{J}^{\;\beta}$ and use 
the following relations,
\[
D_i{\mathcal G}_{LM} = 
\sqrt{\text{--}1} e^K C_{ikl} \xi_L^{\;k} \xi_{M}^{\;l} \;\;,\;\;
D_i \xi_I^{\;\alpha}=-C_{imn}S^{n\alpha}\xi_I^{\; m} 
-\xi_I^{\;0}\delta_i^{\;\alpha} \;\;.
\]
For the fourth relation, we formulate the following two lemmas and 
use the relation (\ref{eqn:Gamma-f}) below.
\end{proof}

\vskip0.1cm

\begin{lemma} \mylabel{lemma:dK1}
\[
 \partial_{\bar k} \xi_I^{\;i}=0 \,,\;\;\; 
\partial_{\bar k}  \xi_I^{\;0} = - \partial_{\bar k} (K_m \xi_I^{\;m}) \;\;.
\]
\end{lemma} 


\begin{proof} \mylabel{lemma:dK}
The matrix $(\xi_{I}^{\;\alpha})$ is the inverse of 
$(\xi_{\alpha}^{\; I})$ by definition. 
From this, we have $\pd_{\bar k} \xi_{I}^{\; \alpha}=
-\xi_{I}^{\; \beta} \pd_{\bar k}\xi_{\beta}^{\;J} \, \xi_{J}^{\alpha}$. 
Now using 
$ \pd_{\bar k} \xi_\beta^{\;J} =
\pd_{\bar k} D_{\beta} X^{J} = g_{\bar k \beta} \xi_{0}^{\; J}$, we obtain 
\[
\pd_{\bar k} \xi_{I}^{\; \alpha} = - 
\xi_{I}^{\; \beta} g_{\bar k \beta} \xi_{0}^{\; J} \, 
\xi_{J}^{\alpha} = - g_{\bar k \beta} \xi_{I}^{\; \beta} \delta_0^{\; \alpha} . 
\]
The first claimed relation is the case when $\alpha=i$. 
The second relation follows from the case 
$\alpha=0$ together with the first relation. 
\end{proof}

\vskip0.1cm

\begin{lemma} Define $f_{ij}^k = (\pd_i\pd_j X^I) \,\xi_I^{\;k}$, 
then $f_{ij}^k$ is holomorphic and we have 
\[
\pd_i K_j - K_i K_j =-C_{ijk}S^k + f_{ij}^m K_m + h_{ij} \;\;,
\]
where $h_{ij}=-(\pd_i \pd_j X^I) h_I$ with a holomorphic function $h_I=h_I(x)$.
\end{lemma}


\begin{proof} \hspace{-4pt} 
When differentiating twice the defining relation of 
the K\"ahler potential 
$e^{-K}=\langle \Omega_x, \bar\Omega_x \rangle = 
\sqrt{\text{--}1}X^I (\overline \tau -\tau)_{IJ} \bar X^J$, we have 
\[
e^{-K}\big(\text{--}\pd_i K_j + K_i K_j \big) = 
\sqrt{\text{--}1}\pd_i\pd_j X^I (\bar\tau-\tau)_{IJ}\bar X^J 
-\sqrt{\text{--}1} \pd_i X^I \pd_j X^L \tau_{ILJ} \bar X^J,
\]
where we set $\tau_{IJK}=\frac{\pd \;}{\pd X^K} \tau_{IJ}$. 
On the other hand, by definition of $S^m$, we have 
\[
\begin{aligned} 
C_{ijm}S^m &= - e^{2K} C_{ijm}(\tau-\bar\tau)_{IJ} 
\bar \xi_{\bar 0}^{\; I} \bar \xi_{\bar m}^{\; J} g^{m\bar m} g^{0\bar 0}  \\
&= -\sqrt{\text{--}1} e^K C_{ijm} {\mathcal G}_{IJ} \bar X^I 
\bar \xi_{\bar m}^J g^{m\bar m}  \\
&= -\sqrt{\text{--}1} e^K C_{ijm} \bar X^{L} \xi_{L}^{\; m} = 
-\sqrt{\text{--}1} e^{K} \tau_{IJL} \pd_i X^I \pd_j X^J \bar X^L \;,
\end{aligned}
\]
where we use $X^M \tau_{IJM}=0$ which follows from the homogeneity 
property of ${\mathcal F(X)}$.  
Using ${\mathcal G}_{IJ}=\xi_I^{\;\alpha} g_{\alpha\bar\beta} 
\bar\xi_{J}^{\bar \beta}$ and 
$\bar\xi_{\bar 0}^J =\bar X^J$, we also have 
\[
\sqrt{\text{--}1}e^K \pd_i \pd_j X^I (\overline \tau -\tau)_{IJ} \bar X^J  \\
= - \pd_i \pd_j X^I {\mathcal G}_{IJ} \bar X^J 
= \pd_i\pd_j X^I \xi_I^{\;0} \;\;, 
\]
where, due to Lemma \ref{lemma:dK1}, we may use 
$\xi_I^{\; 0}= -\xi_{I}^{\; m} K_m + h_I$ 
with some holomorphic function $h_I$. 
Substituting all these relations into the first equation, 
we obtain the claimed formula. 
The holomorphicity of $f_{ij}^k = (\pd_i\pd_j X^I) \xi_{I}^{\; k}$ follows 
from the same Lemma \ref{lemma:dK1}.  
\end{proof}

\vskip0.1cm

From the above lemma, and $\pd_{\bar m} S^{k}=g_{m\bar m} S^{mk}$, 
we obtain 
\begin{equation}
\Gamma_{ij}^k = g^k_i K_j + g^k_j K_i -C_{ijm} S^{mk} + f_{ij}^k \;\;,
\mylabel{eqn:Gamma-f}
\end{equation}
which derives the special K\"ahler relation (\ref{eqn:spK2}) directly 
from the definitions. 

\vskip0.1cm

The connection (\ref{eqn:Gamma-f}) contains non-geometric 
object $f_{ij}^k(x)$, however this does not appear 
in the formulas (\ref{eqn:DS-DK}) since the first three equations 
follows directly from the special K\"ahler relations as we have already seen.  
For the the fourth equation, one 
observes that $f_{ij}^k(x)$ cancels in the evaluation of $D_i K_j$.

\vskip0.1cm

\begin{remark} 
The BCOV ring ${\mathcal R}_{BCOV}^0$ is not finitely generated in general, 
however it is very `close' to this property. This can be observed in 
the following formula
\begin{equation}
\begin{aligned}
D_l C_{ijk} 
&= \sum_{\{a,b\}\cup\{c,d\}={\mathcal I}} 
C_{abm}S^{mn}C_{ncd}+
\tau_{IJKL} \xi_{i}^{\; I} \xi_j^{\;J} \xi_k^{\; K} \xi_{l}^{\;L} \\
&=
\sum_{\{a,b\}\cup\{c,d\}={\mathcal I}} 
C_{abm}S^{mn}C_{ncd}
-\sum_{a\in {\mathcal I}} K_a C_{{\mathcal I}\setminus\{a\}} + h_{ijkl} 
\end{aligned}
\mylabel{eqn:hijkl}
\end{equation}
where we set ${\mathcal I}=\{ i,j,k,l \}$ and 
$h_{ijkl}:=\tau_{IJKL} 
\pd_iX^I 
\pd_jX^J 
\pd_kX^K 
\pd_lX^L$ with $\tau_{IJKL}=
\frac{\pd \;}{\pd X^I} 
\frac{\pd \;}{\pd X^J} 
\frac{\pd \;}{\pd X^K} 
\frac{\pd \;}{\pd X^L} {\mathcal F}(X)$.
Eq.(\ref{eqn:hijkl}) follows from $C_{ijk}=\tau_{IJK}\xi_i^{\; I}
\xi_j^{\;J}\xi_k^{\;K}$ with the relations (\ref{eqn:Dxi-CS}) and  
$D_k \tau_{IJK}=(\pd_k-K_k)\tau_{IJK}= 
\tau_{IJKL}\xi_k^{\;L}$. 
If the holomorphic term $h_{ijkl}$ were zero, then 
the BCOV ring ${\mathcal R}_{BCOV}^0$ reduces to a finitely generated ring. 
This finite generating property can be realized in general by considering 
the following quotient: First, let us note that under the relation 
(\ref{eqn:hijkl}), the BCOV ring may be written as 
\[
{\mathcal R}_{BCOV}^0={\bf Q}[S,S^i,S^{ij},K_i,C_{ijk},\{ h_{ij} \}, 
\{ h_{ijkl} \} ] \;\;,
\]
where $\{h_{ij}\}, \{ h_{ijkl} \}$ means the infinite sequence of the covariant 
derivations of $h_{ij}$ and $h_{ijkl}$, respectively. Considering a 
differential ideal 
${\bf Q}[\{ h_{ij} \}, \{ h_{ijkl} \} ]$, we may reduce 
the ring ${\mathcal R}_{BCOV}^0$ to the quotient
\begin{equation}
{\mathcal R}_{BCOV}^{0,red} = {\mathcal R}_{BCOV}^0 \big/ {\bf Q}
[\{ h_{ij} \}, \{ h_{ijkl} \}] \;\;.
\mylabel{eqn:red-BCOVring}
\end{equation}
We call this quotient ring {\it reduced BCOV ring}. 
\end{remark}

\vskip0.3cm

{\bf (3-2) BCOV ring ${\mathcal R}^\Gamma_{BCOV}$. }
As we have remarked in the previous section, the BCOV ring is defined in the 
algebra of global (mermorphic) sections of the bundle 
(\ref{eqn:sym-tensor}) over the covering space $\tilde{\mathcal M}$. 
Since there is a natural action of the covering group 
$\Gamma \subset Sp(2r+2,{\bf Z})$ on the sections, one may 
consider the invariants under this group action. 
Elements in ${\mathcal R}_{BCOV}^0$, in general, are not invariant under 
this group action, however they can be 'lifted'  
to define $\Gamma$-invariants by specifying a `lift' for each 
propagator (, see below). We then define ${\mathcal R}^\Gamma_{BCOV}$ 
the minimal differential ring of $\Gamma$-invariants which contains 
those $\Gamma$-invariants from ${\mathcal R}_{BCOV}^0$. 
We may call ${\mathcal R}^\Gamma_{BCOV}$ as a $\Gamma$-completion of the BCOV 
ring ${\mathcal R}_{BCOV}^0$.

Let us first note that the generators $K_i(x,\bar x)$ and $C_{ijk}(x)$ are 
invariant under the group $\Gamma$ by their definitions. 
Since the generators $S^{\alpha\beta}$ are transformed according to 
(\ref{eqn:propagator-Tr}), we modify them to $\Gamma$-invariants 
$\tilde S^{\alpha\beta}$. 

Let us assume $\tilde S^{kl}=S^{kl}+\myDelta S^{kl}$ for a $\Gamma$-invariant 
lift. Then it may be determined simply by writing the equation 
(\ref{eqn:Gamma-f}) as 
\begin{equation}
\Gamma_{ij}^k = 
\delta^k_i K_j +\delta^k_i K_j-C_{ijm}\tilde S^{mk} + \tilde f_{ij}^k \;,\quad 
(\tilde f_{ij}^k = f_{ij}^k+C_{ijm}\myDelta S^{mk}) \;,
\mylabel{eqn:Gamma-tildeS}
\end{equation}
and requiring $\Gamma$-invariance of $\tilde f_{ij}^k$. We will show,  
in the example of an elliptic curve, the simplest way to impose the invariance 
is to require  $\tilde f_{ij}^k$ to be a rational function (section) on 
${\mathcal M}$. In the 
original paper by BCOV, this process is referred to as `fixing 
holomorphic (mermorphic) ambiguity'(section 6.3 of \cite{BCOV2}).

Once $\tilde S^{kl}$ is determined in this way, the form of 
other $\Gamma$-invariant propagators $\tilde S^k$, $\tilde S$ may 
be restricted, by requiring the relations 
$\pd_{\bar k}\tilde S^l =g_{k\bar k}\tilde S^{kl}$ and 
$g_{k\bar k}\tilde S^k = \pd_{\bar k}\tilde S$ given 
in Proposition \ref{prop:DS}, to  
\begin{equation}
\tilde S^k = S^k + \myDelta S^{kl} K_l + \myDelta S^k \;\;,\;\; 
\tilde S = S + \frac{1}{2}\myDelta S^{kl} K_k K_l + \myDelta S^k K_k 
+ \myDelta S\;\;,\;\; 
\mylabel{eqn:other-S}
\end{equation}
where $\myDelta S^k$ and $\myDelta S$ are suitable mermorphic sections. 
The form of $\myDelta S^k$ 
can also be determined, in a similar way to (\ref{eqn:Gamma-tildeS}), from
\begin{equation}
\pd_i K_j - K_i K_j = 
-C_{ijk} \tilde S^k + \tilde f_{ij}^m K_m + \tilde h_{ij} \;,\;\;
(\tilde h_{ij}=h_{ij} + C_{ijk}\myDelta S^k) ,
\mylabel{eqn:Gamma-tildeS2}
\end{equation}
by requiring that $\tilde h_{ij}$ is a 
rational function (section) on ${\mathcal M}$.

\vskip0.5cm

\begin{proposition} For the $\Gamma$-invariant propagators, we have 
\begin{equation}
\begin{aligned}
D_i \tilde S^{kl}&=\delta^k_i \tilde S^l+\delta^l_i \tilde S^k-
C_{imn} \tilde S^{mk} \tilde S^{nl} + {\mathcal E}_i^{kl} \;,\\
D_i \tilde S^k\; &=-C_{imn}\tilde S^m \tilde S^{nk} + 2 \delta^k_i \tilde S + 
{\mathcal E}_i^{km} K_m + {\mathcal E}_i^k  \;,\\
D_i \tilde S\;\; &=-\frac{1}{2}C_{imn}\tilde S^m \tilde S^n + 
\frac{1}{2}{\mathcal E}_i^{kl} K_k K_l + {\mathcal E}_i^m K_m + 
{\mathcal E}_i \;,\\
D_i K_j &= -K_i K_j +C_{ijm}\tilde S^{mn}K_n-C_{ijm}\tilde S^m + 
C_{ijm}\kappa^m +\tilde h_{ij}\;,
\end{aligned}
\mylabel{eqn:Diff-tilde-S}
\end{equation}
where we set $\kappa^m=\myDelta S^m$ and  
\[
\begin{aligned}
&{\mathcal E}_i^{kl}=
{\mathcal D}^f_i \myDelta S^{kl}
-\delta^k_i \myDelta S^l -\delta^l_i \myDelta S^k +
C_{imn}\myDelta S^{mk}\myDelta S^{nl},  \\
&{\mathcal E}_i^k =
{\mathcal D}^f_i \myDelta S^k - 2 \delta^k_{i}\myDelta S + 
C_{imn}\myDelta S^n \myDelta S^{nl} 
\;\;,\;\;
{\mathcal E}_i =
{\mathcal D}_i^f \myDelta S +\frac{1}{2} C_{imn}\myDelta S^m \myDelta S^n \;.
\end{aligned}
\]
We also define ${\mathcal D}_i^f:= \pd_i + f_{i *}^*$ a (covariant) derivative 
with $f_{ij}^k(x)$ in (\ref{eqn:Gamma-f}) being treated 
as a connection. 
\end{proposition}


\begin{proof} Use the definitions $\tilde S^{kl}, 
\tilde S^k, \tilde S$, 
$\Gamma_{ij}^k = \delta^{k}_i K_j + \delta^k_j K_i-C_{ijm}S^{mk}$ 
$+f_{ij}^k $ and Proposition \ref{prop:DS-DK} for the evaluations.
After some algebra, the claimed formulas follow. 
\end{proof}

\vskip0.1cm

If we define 
\begin{equation}
\hat{\mathcal E}_i^{kl}= {\mathcal E}_i^{kl} \;,\;\;
\hat{\mathcal E}_i^k = {\mathcal E}_i^{km}K_m + {\mathcal E}_i^k \;,\;\;
\hat{\mathcal E}_i = \frac{1}{2}{\mathcal E}_i^{kl} K_k K_l + 
{\mathcal E}_i^m K_m + {\mathcal E}_i
\mylabel{eqn:E-gens}
\end{equation}
then these are   
sections of weight $(-2,0)$ which are invariant under the action $\Gamma$. 
Considering all covariant derivatives of these tensors, and also $\kappa^m$, 
we have the minimal ring of $\Gamma$-invariants
\[
{\mathcal R}_{BCOV}^\Gamma =
{\bf Q}[\tilde S^{kl},\tilde S^k, \tilde S, K_i, \{C_{ijk}\}, 
\{ \hat{\mathcal E}_i^{kl} \}, 
\{ \hat{\mathcal E}_i^{k} \}, 
\{ \hat{\mathcal E}_i \}, \{ \kappa^m \}, \{ \tilde h_{ij} \} ] \;\;,
\]
where the bracket notation is used for the infinite set of 
the generators as before. 
We should note that the explicit forms of the new generators 
$\hat{\mathcal E}_i^{kl}, \hat{\mathcal E}_i^k, \hat{\mathcal E}_i$ and 
$\kappa^m$ depend on the `lifts' of the propagators 
$\tilde S^{ij}, \tilde S^k, \tilde S$.  Hence the ring 
${\mathcal R}_{BCOV}^\Gamma$ also depends on the lifts.

In case of elliptic curves, we will observe that 
the ring ${\mathcal R}_{BCOV}^\Gamma$ 
has a close similarity to the ring of almost 
holomorphic modular forms studied in Kaneko-Zagier \cite{KZ}.

\vskip0.3cm

{\bf (3-3) Example (elliptic curve).} 
We have introduced the BCOV ring for Calabi-Yau threefolds, however if we 
replace the special K\"ahler geometry by the geometry of upper-half plane, 
it naturally reduces to the rather standard theory 
of (almost holomorphic) modular forms\cite{KZ}. 

Let us consider a family of elliptic curves over ${\mathcal M}$ and 
its period integrals following (2-1).  We consider a family of 
hypersurfaces $Y_{\vec a}$ 
\[
{\tt W}(a):=a_0+a_1 U + a_2 V + a_3 \frac{1}{U^3V^2} =0 \; \subset ({\bf C}^*)^2
\]
in the torus $({\bf C}^*)^2$. Compactifying $({\bf C}^*)^2$ to 
a suitable toric variety ${\bf P}_\Sigma$, we obtain our family of elliptic 
curves. The moduli space ${\mathcal M}$ arises as the parameter space of 
the defining equation. Because of the natural torus actions 
on the parameters, it is easy to see that ${\mathcal M}$ is given by 
${\bf P}^1$, and we have 
\[
\Omega_x = R\Big( \frac{a_0}{{\tt W}(a)} 
\frac{dU}{U}\frac{dV}{V} \Big) \quad 
\big( x = \frac{a_1^3a_2^2a_3}{a_0^6} \in {\bf P}^1).
\]
Taking a symplectic basis $A, B$ with 
$Q=\left( \begin{smallmatrix} 0 & 1 \\ -1 & 0 \\ \end{smallmatrix} \right)$, 
we define the period integral $\vec \omega =(w_0(x),w_1(x))$. The period 
integrals satisfy the Picard-Fuchs differential equation of the form,
\[
\{ \theta_x^2 -12 x (6 \theta_x+1)(6\theta_x+5) \} \omega_i(x) =0 \;\;,
\]
where $\theta_x=x \frac{d \;}{d x}$. 
The `Griffiths-Yukawa coupling' in this case is simply defined by 
\[
C_x := - \int_{Y_x} \Omega_x \cup \frac{d\;}{d x} \Omega_x 
= \frac{1}{(1-432 x)x} \;\;.
\]
Let us fix (uniquely) the $A$ cycle by the condition that the corresponding 
period 
integral $w_0(x)$ is regular at $x=0$ and normalized by 
$\omega_0(x)=1+\cdots$. Then we take a dual cycle $B$ to $A$. With this 
choice of the basis, the period matrix takes the form 
${\bf \Omega}=(\omega_0(x) \, \omega_1(x) ) = \omega_0(x) ( \,1 \; t \,)$ 
with $t\sim \frac{1}{2\pi i} \log x + \cdots $ near $x=0$. We invert the 
relation $t=\frac{\omega_1(x)}{\omega_0(x)}$ as $x=x(t)$. Then it is 
standard to obtain the following identities 
(see eg. \cite{LY});
\begin{equation}
\frac{1}{2\pi i}\frac{1}{\omega_0(x)^2} C_x 
\frac{d x}{d t}  =1 \;\;,\;\;
\omega_0(x(t))^4 = E_4(t)  \;\;,\;\; C_x(x(t))= j(t) \;\;,
\mylabel{eqn:Elliptic-identities}
\end{equation}
where $E_4(t)$ is the Eisenstein series, $j(t)$ is the normalized 
$j$-function with their Fourier 
expansion given by $q=e^{2\pi i t}$ (and 
$\frac{1}{2\pi i} \frac{d \;}{d t}=q\frac{d \;}{d q})$. 
We also have the following useful relation
\begin{equation}
\frac{ \omega_0(x(t))^{12} }{C_x(x(t))} = \eta(t)^{24} \;\;,
\mylabel{eqn:w0-eta}
\end{equation}
in terms of the Dedekind $\eta$-function. Note that the first 
identity of (\ref{eqn:Elliptic-identities}) 
simply represents the fact that there is no quantum correction to 
the Griffiths-Yukawa coupling.

\vskip0.3cm

{\bf (3-3.a)  The BCOV ring ${\mathcal R}_{BCOV}^0={\bf Q}[S, K_x, C_x]$.} 
For an elliptic curve, Definition \ref{def:Sab} of $S=S^{00}$ 
should be read as
\[
S=\frac{1}{2\pi i} g^{0\bar 0}g^{0 \bar 0} 
e^{2K}(t-\bar t)\bar\xi_{\bar 0} \bar\xi_{\bar 0} =
\frac{1}{2\pi i}e^{2K}(t-\bar t) \bar\omega_0 \bar\omega_0 \;.
\]
with the period matrix ${\bf \Omega}=\omega_0 (\,1\, t\,)$. Here, 
for elliptic curves, we introduce the 
factor $\frac{1}{2\pi i}$ in the definition of $S$. 
For the K\"ahler potential, 
we have $e^{-K}=i \; {\bf \Omega} Q \; {}^t \overline{\!{\bf \Omega}}$. 
Then it is straightforward to obtain 
\[
S=\frac{1}{2 \pi i} \frac{1}{\omega_0^2}\frac{1}{\bar t-t} \;\;,\;\;
K_x=\frac{d t}{d x} \frac{1}{\bar t-t} 
-\frac{d\;}{dx} \log(\omega_0(x)) \;\; \;\; .
\]  

\vskip0.1cm

\begin{proposition} The BCOV ring ${\mathcal R}_{BCOV}^0$ is finitely generated 
by $S, K_x$ and $C_x$.  The covariant differential $D_x$ acts on the 
generators by 
\begin{equation}
D_x S = -C_x S S  \,, \;\;
D_x K_x = -K_xK_x-60\, C_x \,,\;\;
D_x C_x = 0 \;\;.
\mylabel{eqn:Dx-BCOV-0}
\end{equation}
\end{proposition} 


\begin{proof} It is sufficient to derive the 
differentials of generators (\ref{eqn:Dx-BCOV-0}). 
The metric connection $\Gamma_{xx}^x$ may be determined from the 
relation; 
\begin{equation}
\begin{aligned}
&(- \partial_x K_x + K_x K_x)e^{-K} \\
&=i \frac{d^2\;}{dx^2} \,{\bf \Omega} \, Q \,{}^t\overline{\!{\bf \Omega}} 
=\Big( -\frac{C_x'}{C_x}K_x+60 \, C_x \Big) e^{-K}\;,
\end{aligned}
\mylabel{eqn:dKx-elliptic}
\end{equation}
where we use the Picard-Fuchs equation 
\[
\frac{d^2\;}{dx^2} \vec \omega  - \frac{C_x'}{C_x} \frac{d\;}{dx} \vec \omega - 
60 C_x \vec\omega = \vec 0 \;\;,
\]
to rewrite $
i \frac{d^2\;}{dx^2} \,{\bf \Omega} \, Q \,{}^t\overline{\!{\bf \Omega}} = 
i \frac{d^2\;}{dx^2} \vec \omega Q \overline{{}^t \vec{\omega}}$ 
as above.  
After differentiating (\ref{eqn:dKx-elliptic}) by $\pd_{\bar x}$, we have 
$\Gamma_{xx}^x  = 2 K_x + \frac{C_x'}{C_x}$. Now using 
(\ref{eqn:dKx-elliptic}) again, we have 
\[
D_x K_x = (\pd_x + \Gamma_{xx}^x)K_x = -K_x K_x -60 C_x \;\;.
\]
Similarly, noting the generators $S, C_x$ have their weights 
$(-2,0)$ and $(2,0)$, 
respectively, and using the relations (\ref{eqn:Elliptic-identities}), 
it is straightforward to obtain the claimed relations.
\end{proof}

\vskip0.1cm

It will be useful to have the following expression for the connection 
$\Gamma_{xx}^x=2K_x+\frac{C_x'}{C_x}$;
\begin{equation}
\Gamma_{xx}^x  
 = 2 \frac{d t}{d x} \frac{1}{\bar t-t} + 
\frac{d \;}{d x} \log \Big( \frac{C_x}{\omega_0(x)^2} \Big)  
= 2 \frac{d t}{d x} \frac{1}{\bar t-t} + 
\frac{d x}{d t} \frac{d \;}{d x} \frac{d t}{d x}   \;,
\mylabel{eqn:Gamma-ellip}
\end{equation}
where we use the identify 
$\frac{d t}{d x} =\frac{1}{2\pi i}\frac{C_x}{\omega_0^2}$ in 
(\ref{eqn:Elliptic-identities}).  In particular, 
writing the first identity of (\ref{eqn:Gamma-ellip}) as 
\[
\Gamma_{xx}^x = 2 C_x S + 
\frac{d \;}{d x} \log \Big( \frac{C_x}{\omega_0(x)^2} \Big)  
= 2 C_x S + f_{xx}^x \;,
\]
one may regard this as the corresponding relation to (\ref{eqn:Gamma-f}).

\vskip0.3cm

{\bf (3-3.b) The BCOV ring ${\mathcal R}_{BCOV}^\Gamma$.}  
In our example, the covering group $\Gamma$ is given by the (genus one) 
modular subgroup $\langle 
\big( 
\begin{smallmatrix} 1 & 0 \\
-1 & 1 \end{smallmatrix} \big), 
\big( 
\begin{smallmatrix} 1 & 1 \\
0 & 1 \end{smallmatrix} \big) \rangle \subset SL(2,{\bf Z})$. 
$S$ is not invariant under the $\Gamma$ 
action, however it is clear from the form $\frac{1}{\bar t - t}$ 
that $S$ can be lifted to a $\Gamma$-invariant by
\[
S \mapsto \tilde S =\frac{1}{\omega_0^2(x)} 
\Big\{ \frac{1}{2\pi i}\frac{1}{\bar t - t} - \frac{1}{12} E_2(t) \Big\} 
\]
in terms of the Eisenstein series $E_2(t)$.  
$\tilde S$ is invariant since $E_2(t)-\frac{1}{2\pi i}
\frac{12}{\bar t - t}=E^*_2(t)$ is the almost holomorphic 
(elliptic) modular form of weight $2$ and 
$\omega_0(x)^2 = \sqrt{E_4(t)}$ for the denominator. 

In our general formulation based on (\ref{eqn:Gamma-tildeS}), 
the invariance arises in a rather weak form as follows: 
We first start with the `shift';
\[
\Gamma_{xx}^x = 2C_x S + f_{xx}^x 
= 2 C_x \tilde S + \tilde f_{xx}^x 
\quad ( \tilde f_{xx}^x:=f_{xx}^x -2 C_x \myDelta S ) \;, 
\]
where $\tilde S=S+\myDelta S$. Accordingly, the formula $D_x S$ changes to 
\[
D_x \tilde S= -C_x \tilde S \tilde S + {\mathcal E}_x 
\quad \Big(
{\mathcal E}_x:=
\pd_x \myDelta S +C_x \myDelta S \myDelta S + 2 \frac{\omega_0'}{\omega_0} 
\myDelta S \Big) \;\;.
\]

\vskip0.1cm

\begin{proposition}  $\tilde f_{xx}^x$ is a rational function of $x$ 
if and only if we set 
\[
\myDelta S= -\frac{1}{C_x} \frac{\omega_0'}{\omega_0} + r(x) \;\;,
\]
with some rational function $r(x)$. In terms of $r(x)$, ${\mathcal E}_x$ is 
given by $r'(x) + C_x r^2(x) -60$. When ${\mathcal E}_x= \lambda C_x$ with 
some constant $\lambda \in {\bf Q}$, 
then the BCOV ring 
${\mathcal R}_{BCOV}^\Gamma$ is finitely generated by 
$\tilde S, K_x, C_x$ with the following differentials, 
\[
D_x \tilde S = -C_x \tilde S \tilde S + \lambda C_x \,, \;\;
D_x K_x = -K_xK_x-60\, C_x \,,\;\;
D_x C_x = 0 \;\;.
\]
\end{proposition}

\begin{proof} 
We evaluate $\tilde f_{xx}^x$ as
\[
\tilde f_{xx}^x=\pd_x \log \big( \frac{C_x}{\omega_0^2} \big) 
-2 C_x \myDelta S 
= 
-2 \frac{\omega_0'}{\omega_0} + \frac{C_x'}{C_x} - 2 C_x \myDelta S \;\;,
\]
from which the first claim is clear. For the evaluation of 
${\mathcal E}_x$, we use the Picard-Fuchs equation satisfied 
the period integral $\omega_0(x)$. The third claim is clear since 
the differentials closes among the generators. 
\end{proof}

\vskip0.1cm

The differential equation 
${\mathcal E}_x = \lambda C_x$ for $r(x)$ may be solved by hypergeometric 
series. From the solution, one may observe that there are infinitely many 
$\lambda$ for which $r(x)$ becomes rational. The simplest result 
is given by 
\[
\lambda=\frac{1}{144} \;,\;\;
r(x)=\frac{1}{12}\frac{C_x'}{C_x^2}\;,\;\;\; \tilde S = -\frac{1}{12}
\frac{E_2^*}{\omega_0^2} \;\;,
\]
where we evaluate 
$S+\myDelta S=S+\frac{1}{12}\frac{1}{C_x}\pd_x \log \frac{C_x}{\omega_0^{12}} 
= S - \frac{1}{12} \frac{1}{C_x}\pd_x \log \eta(t)^{12}$ for $\tilde S$.
Similarly, for $\lambda=\frac{25}{144}, \frac{49}{144}, \frac{121}{144}$, 
for example, we obtain $r(x)=\frac{5}{12}\frac{C_x'}{C_x^2}, 
\frac{1}{12}\frac{C_x'}{C_x^2} -\frac{1}{2}\frac{1}{1-864 x}, 
\frac{5}{12}\frac{C_x'}{C_x^2} -\frac{1}{2}\frac{1}{1-864 x}$ 
and  
\[
\tilde S=  
\frac{\text{--}1}{12 \omega_0^2} 
     \Big\{ E_2^*+4 \frac{E_6}{E_4} \Big\} ,\;
\frac{\text{--}1}{12 \omega_0^2} 
     \Big\{ E_2^*+6 \frac{E_4^2}{E_6} \Big\} ,\;
\frac{\text{--}1}{12 \omega_0^2} 
     \Big\{ E_2^*+4 \frac{E_6}{E_4}+6 \frac{E_4^2}{E_6} \Big\} ,
\]
respectively.

We note that, when the ring is finitely generated, the BCOV ring is very 
close to the ring of almost holomorphic modular forms studied in 
Kaneko-Zagier \cite{KZ}. 
For comparison, it might be useful to write our generators (for the 
case $\lambda = \frac{1}{144}$) in terms of the elliptic modular forms;
\begin{equation}
\tilde S=\frac{\text{--}1}{12}\frac{E_2^*(t)}{\omega_0(x)^2} \,,
\; C_x = j(t) \,,\;
K_x=\frac{\text{--}1}{12}\frac{j(t)}{\omega_0(x)^2}
\Big\{ E_2^*(t)-\frac{E_6(t)}{E_4(t)} \Big\} .
\mylabel{eqn:SK-by-modular}
\end{equation}
One should note, however, that the weight assignment in the BCOV ring is 
different from that of almost holomorphic modular forms. Also, 
in the BCOV ring, we have additional indices of the cotangents 
$\big( \pi^*(T^*{\mathcal M})\big)^{\otimes m}$.

\vskip0.3cm

{\bf (3-4) BCOV ring ${\mathcal R}^{hol}_{BCOV}$. }
For the applications to Gromov-Witten theory of Calabi-Yau manifolds, 
the most relevant form of the BCOV ring is the holomorphic limits of 
the invariants ${\mathcal R}_{BCOV}^\Gamma$, which is 
often referred to as ``$\bar t \rightarrow \infty$'' limit 
in physics literatures. 
For the above example of an elliptic curve, 
the meaning ``$\bar t \rightarrow \infty$'' should be clear as the `limit' 
taking the constant term of $\sum_{n \geq 0} a_m(t) \big( 
\frac{1}{\bar t -t} \big)^n$. We need to formulate a precise meaning for 
Calabi-Yau threefolds. However the idea of the limit 
should be clear from the structure ${\mathcal R}_{BCOV}^\Gamma$ with 
the differentials (\ref{eqn:Diff-tilde-S}). 
Namely, all the differentials are with respect to holomorphic coordinate, 
and therefore ``throwing away'' the 
anti-holomorphic dependence, at the cost of $\Gamma$-invariance, 
should be compatible with the differentiations. 

To describe the holomorphic limit in more detail, let us introduce 
the so-called flat coordinate. We first fix a symplectic basis 
$\{ A^I, B_J \}$ and denote the corresponding period integrals 
$(X^I(x), P_J(x))$. By the property 1) in section 
(2-2), the (half) period maps 
$x \in B (\subset {\mathcal M}) \mapsto [X^I(x)] \in {\bf P}^r$ provides 
a local isomorphism. Due to this property we may introduce 
the so-called {\it flat} coordinate $(t^a)_{a=1,\cdots r}$ by 
the relation 
\[
(X^0(x),X^1(x),\cdots,X^r(x)) = X^0(x) (1,t^1,\cdots,t^r )\;\;,
\]
near  $X^0(x) \not=0$. In this flat coordinate we have for the K\"ahler  
potential 
$
e^{-K(x,\bar x)}= i X^0(x) \overline{X^0(x)} e^{-{\mathcal K}(t,\bar t)} $ 
with 
\[
e^{-{\mathcal K}(t,\bar t)} =2 \overline{ F(t)} - 2 F(t) 
+(t^a-\bar t^a) \big( \frac{\pd F}{\pd t^a} + \overline{ \frac{\pd F}{\pd t^a} } 
\big),
\]
and $F(t)=\frac{1}{(X^0)^2} {\mathcal F}(X) = {\mathcal F}(\frac{X^a}{X^0})$. 
Connections of the bundles in these two 
local coordinates $(x^i)$ and $(t^a)$ are related by 
\[
K_i = -\pd_i \log X^0(x) + \frac{\pd t^a}{\pd x^i} {\mathcal K}_{t^a} \,, \;
\Gamma_{ij}^k = 
\frac{\pd x^k}{\pd t^c} \Gamma_{t^at^b}^{t^c} 
\frac{\pd t^a}{\pd x^i} 
\frac{\pd t^b}{\pd x^j} +  
\frac{\pd x^k}{\pd t^a}\frac{\pd \; }{\pd x^i} \frac{\pd t^b}{\pd x^j} \;.  
\]
As we see in the formula $e^{-{\mathcal K}(t,\bar t)}$, holomorphic 
and anti-holomorphic dependences are not separated by a factor like 
$f(t) g(\bar t)$ (or $f(t)+g(\bar t)$ in $\log{\mathcal K}$). 
We assume that the 'constant terms' 
against to the anti-holomorphic dependences are selected simply by 
setting to zero those expressions  written 
by ${\mathcal K}_{t^a}(t,\bar t)$ and $\Gamma_{t^at^b}^{t^c}$ (and 
also their holomorphic derivatives). 

\vskip0.1cm

\begin{definition} Choose a symplectic basis ${\mathcal B}:=\{ A^I, B_J \}$. 
Then we define the holomorphic limit 
of the elements in ${\mathcal R}_{BCOV}^\Gamma$, with respect to 
${\mathcal B}$,  by the following replacements of the connections: 
\[
K_i \rightarrow {\tt K}_i:= - \pd_i \log X^0(x) \;\;,\;\;
\Gamma_{ij}^k \rightarrow {\tt \Gamma}_{ij}^k := 
\frac{\pd x^k}{\pd t^a} \frac{\pd \; }{\pd x^i} \frac{\pd t^a}{\pd x^j} \;\;.
\]
\end{definition}

\vskip0.1cm

As remarked above, the holomorphic limit commutes with the holomorphic 
differentials $D_i$, and hence we have the same differentials 
as (\ref{eqn:Diff-tilde-S}). 
We denote the holomorphic limit 
of the generators $\tilde S^{ij},\tilde S^k, \tilde S$, respectively by 
${\tt S}^{ij}, {\tt S^k}, {\tt S}$. Also 
by ${\tt D}_i (= \pd_i \pm k {\tt K}_i \pm {\tt \Gamma}_{i*}^{*})$, 
we represent the holomorphic limit of the 
differential $D_i$. Accordingly, ${\tt K}_i$ should be assumed in the 
definitions (\ref{eqn:E-gens}) of $\hat{\mathcal E}_i^{k}, 
\hat{\mathcal E}_i$, although we use the same notation for these.
Thus, taking the holomorphic limit of the $\Gamma$-invariant 
BCOV ring, ${\mathcal R}_{BCOV}^\Gamma$, we will have {\it holomorphic} BCOV 
ring,
\[
{\mathcal R}_{BCOV}^{hol}=
{\bf Q}[{\tt S}^{ij},{\tt S^k}, {\tt S}, {\tt K}_i, \{ C_{ijk} \}, 
\{ \hat{\mathcal E}_i^{kl} \}, 
\{ \hat{\mathcal E}_i^{k} \}, 
\{ \hat{\mathcal E}_i \}, \{ \kappa^m \}, \{ \tilde h_{ij} \} ] \;\;.
\]
The concrete form of the generators ${\tt S}^{ij}$ may be determined 
from the holomorphic limit of the relation 
(\ref{eqn:Gamma-tildeS});
\begin{equation}
{\tt \Gamma}_{ij}^k = \delta^k_i {\tt K}_j + 
\delta^k_j {\tt K}_i -C_{ijm} {\tt S}^{mk} 
+ \tilde f_{ij}^k \;\;.
\mylabel{eqn:topSij}
\end{equation}
Similarly, for ${\tt S}^k$, we can use (\ref{eqn:Gamma-tildeS2}),
\begin{equation}
\pd_i {\tt K}_j -{\tt K}_i {\tt K}_j = 
-C_{ijk} {\tt S}^k + \tilde f_{ij}^m {\tt K}_m + \tilde h_{ij} \;\;. 
\mylabel{eqn:topSij2}
\end{equation}
These equations are used to determine the propagators in \cite{BCOV2}. 
As noted there, when $r \geq 2$, 
the first relation (\ref{eqn:topSij}) provides an 
overdetermined system for ${\tt S}^{ij}$, and the form of 
$\tilde f_{ij}^k$ should be restricted so that there exist 
solutions ${\tt S}^{ij}$. 
Similarly, $\tilde h_{ij}$ should be restricted so that the 
relation (\ref{eqn:topSij2}) has a solution ${\tt S}^k$.
If we find a set of solutions ${\tt S}^{ij}, {\tt S}^k$, the first 
equation of (\ref{eqn:Diff-tilde-S}) determines 
${\mathcal E}_i^{kl}$, and the second relation of (\ref{eqn:Diff-tilde-S}) 
determines ${\tt S}$ up to ${\mathcal E}_i^k$. 
As argued in \cite{BCOV2}, 
the possible forms of rational functions (sections) 
$\tilde f_{ij}^k, 
\tilde h_{ij}$ may be restricted, to some extent, by imposing regularity 
(or singularity) of ${\tt K}_i$ at certain degeneration loci of the family, 
see \cite{BCOV2}).

\vskip0.3cm

\noindent
{\bf Example 1 (elliptic curve):} 
When we take the symplectic basis ${\mathcal B}=\{ A, B \}$ as in the 
previous section, the holomorphic limit of ${\mathcal R}_{BCOV}^\Gamma$ 
is exactly the map taking the constant term of 
$\sum c_m \big( \frac{1}{\bar t-t} \big)^m$. 
From the example in the previous section, 
it is immediate to obtain (for $\lambda=\frac{1}{144}$) that
\[
{\tt S}=-\frac{1}{12} \frac{E_2(t)}{\omega_0(x)^2} \;\;,\;\;
{\tt K}_x = - \pd_x \log (\omega_0(x)) \;\;,\;\;
\Gamma_{xx}^x = \frac{\pd x}{\pd t} \frac{\pd \;}{\pd x} 
\frac{\pd t}{\pd x} \;\;.
\]
As for the differentials, we have the same form as those 
in ${\mathcal R}_{BCOV}^\Gamma$, i.e.,
\[
{\tt D}_x {\tt S} = - C_x {\tt S} {\tt S} + \frac{C_x}{144} \;\;,\;\;
{\tt D}_x {\tt K}_x = - {\tt K}_x {\tt K}_x - 60 \, C_x \;\;,\;\;
{\tt D}_x C_x =0 \;\;.
\]
These relations define the BCOV ring 
${\mathcal R}_{BCOV}^{hol}={\bf Q}[ {\tt S},{\tt K}_x, C_x]$.
The form of the generators are given simply by  
$E_2^*(t) \rightarrow E_2(t)$ in (\ref{eqn:SK-by-modular}).

\vskip0.3cm

\noindent
{\bf Example 2 (mirror quintic Calabi-Yau threefold):} 
The construction of a symplectic basis ${\mathcal B}$ 
about the so-called large complex 
structure limit has been done in \cite{CdOGP} (, see also \cite{Ho1} 
and references therein for its combinatorial construction). 
We consider the holomorphic limit with respect to this basis. 
To fix the propagators ${\tt S}^{ij}, {\tt S}^k$, we have to solve 
the equations (\ref{eqn:topSij}) and (\ref{eqn:topSij2}) 
finding suitable choices for the rational functions 
$\tilde f_{xx}^x, \tilde h_{xx}$. 
In \cite{BCOV2}, it has been found that these unknowns are 
uniquely fixed by requiring expected properties for the higher genus 
Gromov-Witten potential, ${\mathcal F}_g$, which 
comes from the anomaly equation. Here we simply translate their results 
into our conventions.  
First, the propagators ${\tt S}^{xx}, {\tt S}^x$ are determined by 
the choice $\tilde f_{xx}^x = - \frac{8}{5} \frac{1}{x}, 
\tilde h_{ij}=\frac{2}{25}\frac{1}{x^2}$ in 
\[
{\tt \Gamma}_{xx}^x = 2 {\tt K}_x  - C_{xxx}{\tt S}^{xx} 
-\frac{8}{5} \frac{1}{x}   \,, \; 
\pd_x {\tt K}_x -{\tt K}_x {\tt K}_x = 
-C_{xxx}{\tt S}^x -\frac{8}{5} \frac{1}{x} {\tt K}_x
 + \frac{2}{25}\frac{1}{x^2}  ,
\]
where $C_{xxx}=\frac{5}{x^3(1-5^5x)}$ and $x$ is related to $\psi$ in 
\cite{BCOV1}\cite{BCOV2} by $x=\frac{1}{5^5\psi^5}$. 
Then the differentials are evaluated to be 
\[
\begin{aligned} 
{\tt D}_x {\tt S}^{xx}  &=2 \, {\tt S}^x - C_{xxx} {\tt S}^{xx}{\tt S}^{xx} 
+ \frac{x}{25} \;,\\
{\tt D}_x {\tt S}^{x} \;&=2 \, {\tt S} - C_{xxx} {\tt S}^{x}{\tt S}^{xx} 
+ \frac{x}{25}{\tt K}_x - \frac{1}{125} \;,\\
{\tt D}_x {\tt S} \;\;  &=-\frac{1}{2}C_{xxx}{\tt S}^x{\tt S}^x 
+ \frac{x}{50} {\tt K}_{x}{\tt K}_{x} 
- \frac{1}{125}{\tt K}_x + \frac{2}{3125} \frac{1}{x} \;,\\
{\tt D}_x {\tt K}_x &= - {\tt K}_x {\tt K}_x + C_{xxx} {\tt S}^{xx}{\tt K}_x - 
C_{xxx}{\tt S}^x + \frac{2}{25}\frac{1}{x^2} \;,\\
\end{aligned}
\]
from which we read ${\mathcal E}_x^{xx}=\frac{x}{25}, 
{\mathcal E}_x^x = -\frac{1}{125}, 
{\mathcal E}_x = \frac{2}{3125}\frac{1}{x}$ and 
$\kappa^x=\frac{2}{25}\frac{1}{x^2}\frac{1}{C_{xxx}}$. 
With these, the BCOV ring is determined by 
\[
{\mathcal R}_{BCOV}^{hol} ={\bf Q}[{\tt S}^{xx},{\tt S}^x, {\tt S}, 
{\tt K}_x, \{ C_{xxx} \}, 
\{ \hat{\mathcal E}_x^{xx} \}, 
\{ \hat{\mathcal E}_x^{x} \}, 
\{ \hat{\mathcal E}_x \}, \{ \kappa^x \} ]  \;\;.
\]

\vskip0.3cm

\noindent
{\bf Example 3:} The Calabi-Yau manifolds whose higher genus 
Gromov-Witten invariants are studied in \cite{HK} are not complete 
intersections in toric varieties, but has an interesting property: there 
exist two different large complex structure limits (cusps) in the 
deformation space\cite{Ro}. By mirror symmetry, this phenomenon is related to 
non-birational Calabi-Yau manifolds whose derived categories of 
coherent sheaves are equivalent \cite{Ro},\cite{BC},\cite{Ku}. 
The cusps of the example in \cite{Ro},\cite{HK} are located at 
$x=0$ and $z=\frac{1}{x}=0$. The BCOV ring ${\mathcal R}_{BCOV}^{hol}$ 
with respect to a symplectic basis ${\mathcal B}_0$ at $x=0$ 
has a similar form as in Example 2 with 
\[
{\mathcal E}_x^{xx}= \frac{-1}{14}\frac{x p(x)}{(x-3)^2} \,,\;
{\mathcal E}_x^{x}= \frac{1}{14}\frac{p(x)}{(x-3)^2} \,,\;
{\mathcal E}_x= \frac{-1}{28}\frac{p(x)(x+14)+q(x)}{(x-3)^3} \,,
\]
and $\kappa^x = \frac{2}{x^2}\frac{1}{C_{xxx}}$, 
where $p(x)=x^4-716x^3+422 x^2 +452 x-15$, $q(x)=
12374 x^3-7166 x^2-7630 x+246$. The BCOV ring ${\mathcal R}_{BCOV}^{hol}$ 
at the other cusp ($z=0$) is defined  
with respect to a different symplectic basis ${\mathcal B}_{\infty}$. However 
we verify from the results in \cite{HK} that the ring is determined with 
${\mathcal E}_z^{zz}= {\mathcal E}_x^{xx} \big(\frac{dz}{dx}\big)$, 
${\mathcal E}_z^{z}= {\mathcal E}_x^x$, 
${\mathcal E}_z = {\mathcal E}_x \big(\frac{dx}{dz}\big)$ and 
$\kappa^z = \kappa^x \big(\frac{dz}{dx}\big)$. From this, we observe that the 
BCOV ring ${\mathcal R}_{BCOV}^\Gamma$ is invariant under the symplectic 
transformation which connects ${\mathcal B}_0$ and ${\mathcal B}_\infty$.

\vskip0.5cm

\section{{\bf BCOV holomorphic anomaly equation in ${\mathcal R}_{BCOV}^{0,red}$}}

\vskip0.1cm

{\bf (4-1) BCOV holomorphic anomaly equation. } 
The original form of the BCOV anomaly equation has been formulated based 
on the special K\"ahler geometry over the moduli 
space ${\mathcal M}$. Although mathematical ground of this anomaly 
equation has not yet been established, up to now, 
this equation provides the only way to a systematic calculation of 
higher genus Gromov-Witten potential ${\tt F}_{g}(t)$ for Calabi-Yau 
complete intersections and some cases beyond them.  
About this equation, recently, several 
important progress has been made in physics literatures \cite{YY},\cite{ABK},
\cite{HKQ},\cite{HK},\cite{AL}. 
In particular, the polynomial property found by Yamaguchi and Yau \cite{YY} 
and also in \cite{AL} is the one which we followed for our 
definition of the BCOV ring ${\mathcal R}_{BCOV}^\Gamma$. 

To summarize the recursive procedure given in \cite{BCOV2}, let us write the 
anomaly equation in the following form, which appeared in \cite{YY} and 
\cite{AL},  
\begin{equation} 
\begin{aligned}  
\frac{\pd {\mathcal F}^{(g)}}{\pd \tilde S^{ij}} &= 
\frac{1}{2} D_i D_j {\mathcal F}^{(g-1)} + 
\frac{1}{2} \sum_{h=1}^{g-1} 
D_i {\mathcal F}^{(g-h)} D_j {\mathcal F}^{(h)} \;\;,\\
0 &= 
\tilde S^{jk} \frac{\pd {\mathcal F}^{(g)}}{\pd \tilde S^k} +
\tilde S^{j} \frac{\pd {\mathcal F}^{(g)}}{\pd \tilde S} +
\frac{\pd {\mathcal F}^{(g)}}{\pd \tilde K_j} \;\; \;\;  \quad (g\geq 2). \\ 
\end{aligned}
\mylabel{eqn:BCOV-g}
\end{equation}
Then the recursion proceeds as follows, with $\tilde f_{ij}^k, \tilde h_{ij}, 
{\mathcal E}_i^k$ in (\ref{eqn:Gamma-tildeS}), (\ref{eqn:Gamma-tildeS2}),
(\ref{eqn:Diff-tilde-S}), respectively, being unknown:

\vskip0.2cm
\noindent
Step 1. We start with the fact that there exists a polynomial  
$f_0(x)$ and a choice $\tilde f_{ij}^k(x)$ such that 
\[
D_i {\mathcal F}^{(1)} = \frac{1}{2} C_{imn} \tilde S^{mn}  
-\big( \frac{\chi}{24}-1\big) K_i + f_{1,i}(x)  \;, \;\;
( f_{1,i} := \pd_i \log f_0 )
\]
gives the genus one Gromov-Witten potential ${\tt  F}_{1}(t)$ of the 
mirror Calabi-Yau manifold $X$ (with its Euler number $\chi$) when 
we take the holomorphic limit. We refer \cite{BCOV1}\cite{BCOV2} 
for details of 
${\tt  F}_{1}(t)$. The polynomial $f_0(x)$ is essentially given by the 
discriminant of the family. We define, using the bracket notation,   
\[
{\mathcal R}_{BCOV}^{\Gamma, 1} = 
{\mathcal R}_{BCOV}^{\Gamma}[\{ f_{1,i}(x) \} ] \;\;, 
\]
and regard $D_i {\mathcal F}^{(1)}$ (and $f_{1,i}$) as an element 
of weight zero in ${\mathcal R}_{BCOV}^{\Gamma, 1}$. 

\vskip0.2cm
\noindent
Step 2. 
Suppose we have $D_i {\mathcal F}^{(1)}$ and 
${\mathcal R}_{BCOV}^{\Gamma, 1}$ as above. Consider the anomaly 
equation (\ref{eqn:BCOV-g}) for $g=2$ in the ring 
${\mathcal R}_{BCOV}^{\Gamma,1}$ to find a (unique) solution 
${\mathcal F}^{(2)}_0$ of weight $(-2,0)$ under the condition 
${\mathcal F}^{(2)}_0 \big|_{\tilde S^{ij}=\tilde S^k = \tilde S =0} =0$.
Then the observation made in \cite{BCOV2} is that there exist a 
rational section $f_2(x)$ of $({\mathcal L}^{-1})^{\otimes 2}$ and suitable 
choices $\tilde h_{ij}(x)$ and ${\mathcal E}_i^k(x)$ such that 
\[
{\mathcal F}^{(2)} = {\mathcal F}^{(2)}_0 + f_2(x) \;\;,
\]
gives the Gromov-Witten potential ${\tt F}_2(t)$ under the holomorphic limit. 
($\tilde f_{ij}^k(x)$ in step 1 and $\tilde h_{ij}(x), {\mathcal E}_i^k(x)$ 
in step 2 fix the lifting $S^{\alpha\beta}$ to $\tilde S^{\alpha\beta}$, and 
thus ${\mathcal R}_{BCOV}^\Gamma$.) 
We extend our BCOV ring to 
${\mathcal R}_{BCOV}^{\Gamma, 2} \hspace{-3.5pt}= 
{\mathcal R}_{BCOV}^{\Gamma, 1}[\{ f_2(x) \}]$. 

\vskip0.3cm 
For $g\geq 3$, this procedure continues genus by genus enlarging 
the BCOV ring by some rational section $f_g(x)$ of the line bundle 
${\mathcal L}^{2-2g}$ to 
${\mathcal R}_{BCOV}^{\Gamma, g+1} =$  
${\mathcal R}_{BCOV}^{\Gamma, g}[\{f_{g}\}]$. When we define a notation 
\[
{\mathcal R}_{BCOV}^{\Gamma, \infty} = 
\lim_{\rightarrow} {\mathcal R}_{BCOV}^{\Gamma, g} \;\;, 
\]
the solutions ${\mathcal F}_g$ are (scalar) elements in this 
ring of weight $(2-2g,0)$. 

\vskip0.3cm

\begin{remark} 
In general, the ring ${\mathcal R}_{BCOV}^{\Gamma,\infty}$ is 
a large ring. However, we note that, since  
${\mathcal R}_{BCOV}^{\Gamma,\infty}$ consists of 
$\Gamma$-invariants,  
working in an affine coordinate $U \subset {\mathcal M}$ of the toric 
variety ${\mathcal M}={\bf P}_{Sec(\Sigma)}$ makes sense. 
Let us consider an affine coordinate of $U$ and consider the field 
of rational functions ${\bf Q}(x)$ on $U$. We assume that 
the unknowns ${\mathcal E}_i^{kl}, {\mathcal E}_i^k, {\mathcal E}_i, 
\kappa^m$ are rational functions on $U$ (as seen in Example 2 and 3 of (3-4)). 
Then, due to the rationality 
of $C_{ijk}$ and $f_g(x)$, we observe 
\begin{equation}
{\mathcal R}_{BCOV}^{\Gamma,\infty} \big|_U  \subset  
{\bf Q}(x)[\tilde S^{ij}, \tilde S^k, \tilde S, K_i] \;\;.
\mylabel{eqn:restrict-U}
\end{equation}
The ring in the r.h.s. of the inclusion is the local form which 
we see the BCOV ring ${\mathcal R}_{BCOV}^{\Gamma,\infty}$ in physics 
literatures, for example \cite{YY}, \cite{HKQ}, \cite{HK}. 
In reference \cite{HKQ}, 
in particular, following the idea of \cite{St2},\cite{GV}, an efficient way to 
impose certain boundary conditions to determine $f_g(x)$ has been 
found.  
\end{remark}

\vskip0.3cm
{\bf (4-2) BCOV anomaly equation in ${\mathcal R}_{BCOV}^{0,red}$.} 
As briefly sketched above, solving BCOV anomaly equation (\ref{eqn:BCOV-g}) 
contains a process finding a suitable $f_g(x)$ at each genus. 
This is the main problem 
to determine the Gromov-Witten potential ${\tt F_g}$ 
from the anomaly equation. Apart from this important 
problem, we can extract some algebraic (combinatorial) structure of the 
equation by considering the same BCOV anomaly equation (\ref{eqn:BCOV-g}) 
in the reduced ring ${\mathcal R}_{BCOV}^{0,red}$ defined in 
(\ref{eqn:red-BCOVring}).

Let us first note that ${\mathcal R}_{BCOV}^{0,red}$ is generated by 
$S^{ij}, S^k, S, K_i$ and $C_{ijk}$ over ${\bf Q}$ in the quotient ring,   
with the `reduced' differential $D_i$ (see Remark 3.9). 
Hereafter all manipulations should be understood in this quotient ring, 
although we abuse the same notations. 
The BCOV anomaly equation has the same form 
as (\ref{eqn:BCOV-g}) with obvious replacements of the generators, 
e.g. $\tilde S^{ij}$ by $S^{ij}$. 
Then the following property is due to \cite{AL}:

\vskip0.1cm

\begin{proposition} Define new generators by 
\[
\hat S^{ij}=S^{ij} \;\;,\;\;
\hat S^k=S^k -S^{km}K_m \;\;,\;\;
\hat S =S-S^m K_m + \frac{1}{2}S^{mn}K_m K_n \;\;,
\]
then the second equation of (\ref{eqn:BCOV-g}) implies 
simply $\frac{\pd {\mathcal F}^{(g)}}{\pd K_m} =0$, namely 
\[
{\mathcal F}^{(g)} \in {\bf Q}[\hat S^{ij}, \hat S^{k}, 
\hat S, C_{ijk}] \subset {\bf Q}[S^{ij}, S^k, S, K_m,C_{ijk}] \;  
( ={\mathcal R}_{BCOV}^{0,red} ) \;\;.
\]
\end{proposition}

\vskip0.1cm

Using the new generators above, the l.h.s of the first equation of 
(\ref{eqn:BCOV-g}) may be written as 
\[
\frac{\pd {\mathcal F}^{(g)}}{\pd \hat S^{ij}} - 
K_j \frac{\pd {\mathcal F}^{(g)}}{\pd \hat S^{i}} + 
\frac{1}{2} K_i K_j \frac{\pd {\mathcal F}^{(g)}}{\pd \hat S} \;\;,
\]
while the r.h.s. of that equation has the following expansion,
\begin{equation}
\begin{aligned}
&\frac{1}{2} D_i D_j {\mathcal F}^{(g-1)} + 
\frac{1}{2} \sum_{h=1}^{g-1} D_i {\mathcal F}^{(g-h)} D_j {\mathcal F}^{h} \\
&= Q_{ij}^{(g-1)} + Q_{i}K_j^{(g-1)} + Q_{j}^{(g-1)}K_i + 
\frac{1}{2} Q^{(g-1)} K_i K_j \;\;.
\end{aligned}
\mylabel{eqn:Q-def}
\end{equation}
Here one should note that by the above Proposition, the dependence 
on $K_i$ comes only from the covariant derivatives. 
Now comparing each coefficient of $1, K_i$, $K_iK_j$, 
we have,
 
\vskip0.1cm

\begin{proposition} The BCOV anomaly equation in ${\mathcal R}_{BCOV}^{0,red}$ 
is equivalent to the following first order system of linear 
differential equations;
\begin{equation}
\frac{\pd {\mathcal F}^{(g)}}{\pd \hat S^{ij}} = Q_{ij}^{(g-1)} \;,\; 
\frac{\pd {\mathcal F}^{(g)}}{\pd \hat S^{i}}  = -Q_i^{(g-1)}  \;,\;
\frac{\pd {\mathcal F}^{(g)}}{\pd \hat S} = Q^{(g-1)}   \;\;(g \geq 2). 
\mylabel{eqn:BCOV-linear}
\end{equation}
With the initial data $Q^{(1)}_{ij}, Q^{(1)}_i, Q^{(1)}$ 
which follow from (\ref{eqn:Q-def}) with 
$D_i {\mathcal F}^{(1)}$ $=\frac{1}{2} C_{ijk} \hat S^{jk}$ $ 
-\big( \frac{\chi}{24}-1\big)K_i$, 
this equation has a unique solution 
${\mathcal F}^{(g)} \in {\bf Q}[\hat S^{ij}, \hat S^{k}, \hat S,C_{ijk}]$ 
of weight $(2-2g,0)$.
\end{proposition}

\vskip0.1cm

The uniqueness of the solution above follows from the weight 
consideration for the possible `constants of 
integration' in the ring ${\mathcal R}_{BCOV}^{0,red}$.

\vskip0.1cm

For the application to Gromov-Witten potential ${\tt F}_g(t)$, 
the BCOV anomaly equation should be 
considered in the ring ${\mathcal R}_{BCOV}^{\Gamma,\infty}$, as 
we have summarized briefly in the previous section.  
However, the simple structure (\ref{eqn:BCOV-linear}) 
extracted above in 
${\mathcal R}_{BCOV}^{0,red}$ is  
still valid for the BCOV anomaly equation in the ring 
${\mathcal R}_{BCOV}^{\Gamma,\infty}$. 
In fact, the form (\ref{eqn:BCOV-linear}) of the BCOV equation has appeared  
first in [\cite{HK}, section (3-4)] to make the solutions 
${\mathcal F}^{(g)}$ (${\tt F}_g(t)$).

\vskip0.5cm

\section{{\bf Conclusions and discussions}}

\vskip0.1cm

After a self-contained introduction to the special K\"ahler geometry, we have 
introduced the differential ring ${\mathcal R}_{BCOV}^0$, which is 
geometric in nature. Combined with the modular property, we considered the 
$\Gamma$-invariant `lifts' $\tilde S^{ij},\tilde S^k,\tilde S$ of the 
propagators. The `lifting' process has been identified with that of fixing 
`mermorphic ambiguities' in \cite{BCOV2}. With a choice of $\Gamma$-invariant 
lifts of the propagators, we defined the ring ${\mathcal R}_{BCOV}^\Gamma$, 
which depends on the choice of the lifts. After taking a symplectic basis 
${\mathcal B}$, we defined the holomorphic limit ${\mathcal R}_{BCOV}^{hol}$ 
following\,\cite{BCOV2}. 
In case of an elliptic curve, we have shown a close relation of our BCOV 
rings to the theory of quasi-modular forms due to Kaneko-Zagier\,\cite{KZ}.  

Considering a suitable quotient, we have reduced the ring 
${\mathcal R}_{BCOV}^0$ 
to a finitely generated differential ring 
${\mathcal R}_{BCOV}^{0,red}$. 
In this reduced ring, we have extracted a simple algebraic  
structure of the BCOV holomorphic anomaly equation which still exists 
before the reduction.

\vskip0.3cm

As briefly summarized in section 4, our construction of the BCOV ring is an 
abstraction of the important progress made in \cite{YY} and 
\cite{AL} for the solutions 
of BCOV holomorphic anomaly equation. 
In 1999, in case of a rational elliptic 
surface $\frac{1}{2}K3$, M.-H.Saito, A. Takahashi and the present 
author\cite{HST} found a similar recursion formula for 
Gromov-Witten potentials, 
\begin{equation}
\frac{\pd Z_{g;n}}{\pd E_2} = 
\frac{1}{24} \sum_{
\begin{subarray}{c} g'+g'' =g \\   g',g''  \geq 0 
\end{subarray} } \sum_{s=1}^{n-1} 
s(n-s) Z_{g';s} Z_{g'';n-s} + 
\frac{n(n+1)}{24} Z_{g-1;n} \;\;,
\mylabel{eqn:Modular-ano}
\end{equation} 
in terms of quasi-modular forms 
$Z_{g;n} \text{=} P_{g;n}\frac{q^{\frac{n}{2}}}{\eta(\tau)^{12n}}$ 
($P_{g;n} \in {\bf Q}[E_2,E_4,E_6]$) with the initial data 
$Z_{0;1}=\frac{q^{\frac{1}{2}} E_4}{\eta(\tau)^{12}}$ (, this  
generalized a previous result in \cite{MNW},\cite{MNVW} 
for $g=0$ case).  Later it has been 
conjectured that the above recursion relation (\ref{eqn:Modular-ano}) is 
equivalent to the BCOV holomorphic anomaly equation  evaluating 
${\mathcal F}^{(g)}$ for $g \leq 3$ \cite{Ho2}. Due to recent progress made 
in \cite{YY} and \cite{AL}, we have now the BCOV anomaly equation of the form, 
\[
\frac{\pd {\mathcal F}^{(g)}}{\pd S^{yy}} = 
\frac{1}{2} \Big( \sum_{h=1}^{g-1} 
D_y {\mathcal F}^{(g-h)} D_y {\mathcal F}^{(h)} 
+ D_y D_y {\mathcal F}^{(g-1)} \Big) \;\;, 
\]
which is defined over a suitable BCOV ring 
${\mathcal R}_{BCOV}^{\Gamma,\infty}$. 
In this form, we can prove the equivalence of the `modular anomaly equation' 
(\ref{eqn:Modular-ano}) to the BCOV holomorphic anomaly equation. It is  
almost clear that the close relationship between the ring of the 
quasi-modular forms and our BCOV ring presented in section 3 plays a central 
role for the equivalence. The detailed results will be reported elsewhere 
\cite{HS}.

\vskip1cm

\end{document}